\documentclass[final]{siamart190516} 
\usepackage[utf8]{inputenc}
\usepackage{cite}
\usepackage{amsmath,amssymb,amsfonts,amsopn}
\usepackage{color}
\usepackage{todonotes}
\usepackage{colonequals}
\usepackage{microtype}

\newcommand{\bbK}{\ensuremath{\mathbb{K}}}  
\newcommand{\bbS}{\ensuremath{\mathbb{S}}}  
\newcommand{\bbX}{\ensuremath{\mathbb{X}}}  
\newcommand{\bbU}{\ensuremath{\mathbb{U}}}  
\newcommand{\bbD}{\ensuremath{\mathbb{D}}}  
\newcommand{\bbT}{\ensuremath{\mathbb{T}}}  
\newcommand{\RR}{\ensuremath{\mathbb{R}}}   
\newcommand{\CC}{\ensuremath{\mathbb{C}}}   
\newcommand{\HM}{\ensuremath{\mathbb{H}}}   
\newcommand{\MW}{\ensuremath{\mathcal{W}}}  
\newcommand{\MU}{\ensuremath{\mathcal{U}}}  
\newcommand{\MV}{\ensuremath{\mathcal{V}}}  
\newcommand{\MS}{\ensuremath{\mathcal{S}}}  
\newcommand{\MF}{\ensuremath{\mathcal{F}}}  
\newcommand{\MG}{\ensuremath{\mathcal{G}}}  
\newcommand{\MA}{\ensuremath{\mathcal{A}}}  
\newcommand{\ME}{\ensuremath{\mathcal{E}}}  

\newcommand{\ddd}[1]{\,{\rm d}#1}  
\newcommand{\dx}{\ddd{x}}  
\newcommand{\dk}{\ddd{k}}  
\newcommand{\dt}{\ddd{t}}  
\DeclareMathOperator{\divergence}{div}

\DeclareMathOperator{\supp}{supp}
\DeclareMathOperator{\Tr}{Tr}
\newtheorem{remark}[theorem]{Remark}

\title{On the unique solvability of radiative transfer equations with polarization}
\author{Vincent Bosboom%
\thanks{Department of Applied Mathematics, University of Twente,
P.O. Box 217, 7500 AE Enschede, The Netherlands.
\email{v.bosboom@utwente.nl}}
\and Matthias Schlottbom%
\thanks{Department of Applied Mathematics, University of Twente,
P.O. Box 217, 7500 AE Enschede, The Netherlands.
\email{m.schlottbom@utwente.nl}}
\and Felix L.~Schwenninger%
\thanks{Department of Applied Mathematics, University of Twente,
P.O. Box 217, 7500 AE Enschede, The Netherlands and Department of Mathematics,
University of Hamburg, Germany. 
\email{f.l.schwenninger@utwente.nl}}
}
\date{\today}
\headers{Well-posedness of radiative transfer equations with polarization}{V. Bosboom, M. Schlottbom, F. Schwenninger}

\begin{document}

\maketitle

\begin{abstract} 
We investigate the well-posedness of the radiative transfer equation with polarization and varying refractive index. 
The well-posedness analysis includes non-homogeneous boundary value problems on bounded spatial domains, which requires the analysis of suitable trace spaces.
Additionally, we discuss positivity, Hermiticity, and norm-preservation of the matrix-valued solution.
As auxiliary results, we derive new trace inequalities for products of matrices.
\end{abstract}







\section{Introduction}
In this paper, we study the well-posedness of the radiative transfer equation (RTE) describing the evolution of polarized radiation:
\begin{align}\label{eq:rte_W}
    \frac{\partial \MW}{\partial t} + \nabla_kH\cdot\nabla_x \MW - \nabla_x H\cdot\nabla_k \MW + \Sigma \MW=N(\MW) + S(\MW) + \MF.
\end{align}
This equation was first derived using phenomenological arguments by Chandrasekhar \cite{Chandrasekhar1960}, but has also been derived from the high-frequency limit of Maxwell's equations \cite{Papanicolaou1975,Ryzhik:1996,Gerard:1997}.
It is used in many applications, such as in medical imaging \cite{ArridgeSchotland:2009}, neutron transport \cite{Case1967},
atmospheric science \cite{Hansen_1974}, oceanography \cite{Arnush_1972}, pharmaceutical powders \cite{Burger:97} or solid state lightning \cite{Melikov_2018}, see also \cite{Bal:2009,Carminati_2021}.

In \eqref{eq:rte_W}, the coherence matrix $\MW=\MW(x,k,t)$, with $x\in \bbX\subseteq\RR^3$, $k\in \bbK=\RR^3\backslash\{0\}$ and $t\geq 0$, is a Hermitian $2\times2$ matrix-valued function and the quantity $H$ refers to a dispersion relation given by $H(x,k)=\nu(x)|k|$, with the velocity function  $\nu(x)$. The coupling operator $N$ is defined as
\begin{align}
    \label{eq:Coupling}
    N(\MW) \colonequals n(J\MW-\MW J),
\end{align}
where the function $n=n(x,k)$ is a scalar-valued coupling function related to the rate of change in polarization, and the symplectic matrix $J$ is defined as
\[ J \colonequals \begin{pmatrix}0&1\\-1&0\end{pmatrix}.\]
The scattering rate, which is associated to scattering by random perturbations in the background medium, is described by the function $\Sigma=\Sigma(x,k)$, while a redistribution of propagation direction is described by the integral operator
\begin{align}
    \label{eq:Scattering}
    (S\MW)(x,k,t) \colonequals \int_{\bbS_{|k|}} \sigma(x,k\cdot k') T(k,k') \MW(x,k',t)T(k,k')^* \,d\lambda(k').
\end{align}
Here, $\bbS_{|k|}$ is the sphere in $\RR^3$ of radius $|k|$, $\sigma\colon \bbX\times\RR\to \RR$ is a positive function and $T(k,k')$ is a real $2\times2$ matrix-valued function satisfying $T(k,k')=T(k',k)^*$, with the $^*$ operator denoting the conjugate transpose. Furthermore, the $2\times 2$ matrix-valued function $\MF=\MF(x,k,t)$ denotes an internal source term. For more details on the parameters $T$ and $n$ we refer the reader to \cite{Ryzhik:1996}.

In this manuscript we prove well-posedness
of the radiative transfer equation with polarization. Moreover, we study additional properties of the matrix-valued solution, such as positivity and dissipation of energy. To do so, we first show that without coupling and scattering the radiative transfer equation can be written in terms of the Liouville equation for the real-valued Stokes parameters $I$, $Q$, $U$, and $V$ \cite{Born:1999}, which are related to the coherence matrix by
\begin{align}\label{eq:Stokes}
\MW=\frac{1}{2}\begin{pmatrix}I + Q & U + i V\\ U - i V & I-Q \end{pmatrix}.
\end{align}
%
We show that under mild assumptions the Liouville equation for our Hamiltonian $H$ is well-posed in $L^p$, $p\in[1,\infty)$, on general domains. The cases $p=1$ and $p=2$ have been investigated before on $\RR^{2n}$, see  \cite{Petrina_1990} and  \cite{Jiang:1997}. 
Secondly, we prove that the coupling and scattering terms define bounded perturbations in the semigroup formalism, leading to the existence and uniqueness of a solution to the full radiative transfer equation \cref{eq:rte_W} in \cref{thrm: ExistenceFull}. 
This analysis requires a careful choice of the norm for the matrix-valued function $\MW$. By using Schatten norms \cite{Bhatia1997} and proving inequalities for the traces of matrices, which seem to be novel and of interest in their own right, we can show that the term $S(\MW)-\Sigma\MW$ is not only related to a bounded, but also to a dissipative operator. Additionally, we show that the term $N(\MW)$ does not change the norm of $\MW$ along the flow. The combination of the latter statements is then used to show that the norm of $\MW$ is decaying in time.
Using the mapping properties of the operators associated to \cref{eq:rte_W} and Trotter's formula, we are able to prove that $\MW(x,k,t)$ remains positive matrix provided the data terms are positive.
In addition, we also investigate the case of a bounded spatial domain and non-homogeneous boundary conditions.

On a very simplified level, our approach can be seen as similar to the program outlined in  Dautray \& Lions, see \cite[Chapter XXI]{DL6}, where  the constant velocity case $\nu(x)={\rm const}$ is treated. 
However, since $H$ is spatially dependent and due to the fact that Equation \cref{eq:rte_W} is matrix-valued, one encounters significant challenges when trying to apply known techniques.
For instance, the scattering operator $S$ requires a more refined a study due to the interplay of Schatten and $L^{p}$ norms. 
Furthermore, we will prove new trace theorems for the case of variable velocity $\nu(x)$, 
enabling us to employ the usual homogenization of boundary conditions argument to verify well-posedness of the radiative transfer equation on bounded domains. To show that this solution is positive pointwise in $(x,k,t)$, we combine Trotter's formula with a positivity result for inhomogeneous boundary conditions. The latter is proven by considering the evolution equations satisfied by the trace and the determinant of $\MW$.

The outline of this paper is as follows: After introducing relevant functions spaces and a list of standing assumptions in \Cref{section: Preliminaries},  we rewrite the radiative transfer equation without the coupling, scattering and source terms as a system Liouville equations for the Stokes parameters in \Cref{section: LiouvilleUnbounded}. The equations on the full space are solved and related to the theory of semigroups by standard arguments. In \Cref{section: CoupleScatter} we show the well-posedness of the full radiative transfer equation with polarization by treating the coupling and scattering terms as bounded perturbations of the semigroup derived in the previous section. We also state and prove novel inequalities for traces of matrices, which might be interesting on their own.
In \Cref{section: bounded-domain} we extend our for the Liouville equation to the case of a bounded spatial domain with non-homogeneous inflow boundary conditions by defining a suitable trace space and lifting operator. The analysis of \cref{eq:rte_W} on bounded domains is done in \Cref{sec:existence_full_rte_bounded}. \Cref{section: Additional} contains statements and proofs regarding positivity of the solution and dissipation of energy. We end with a conclusion and discussion of our work in \Cref{section: conclusion}. 

\section{Preliminaries}
\label{section: Preliminaries}

\subsection{Function spaces}
We denote by $|\MU| \colonequals \left(\MU^*\MU\right)^\frac{1}{2}$ the modulus of a complex matrix $\MU\in\CC^{n\times n}$. The trace of $\MU$ is denoted by $\Tr(\MU)$, and we write $\MU\succeq 0$ if $\MU$ is Hermitian and positive. Furthermore, denote by $\HM$ the space of $2\times 2$ Hermitian matrices. 
We equip $\HM$ with the Schatten $p$-norm \cite{Bhatia1997}, which, for $\MU\in \HM$ and $1\leq p<\infty$, is defined as
\begin{align}
    \|\MU\|_{S^p} \colonequals \Tr\left[|\MU|^p\right]^{\frac{1}{p}} = \left(\sum_{i=1}^2\sigma_i^p(\MU)\right)^\frac{1}{p},
\end{align}
with $\sigma_i(\MU)$ denoting the singular values of $\MU$ with $\sigma_1(\MU)\geq \sigma_2(\MU)$. For $p=\infty$, we set
\begin{align}
    \|\MU\|_{S^\infty} \colonequals \sigma_1(\MU).
\end{align}
For a domain $\bbD\subset\RR^n$, $1\leq p <\infty$ and a (finite-dimensional) complex normed space $V$, let $L^p(\bbD,V)$ denote 
the common Lebesgue-Bochner spaces of $p$-integrable functions from $\bbD$ to $V$. 
The space  $L^p(\bbD,\HM)$ is thus equipped with the norm
\begin{align}
    \|\MU\|_{L^p(\bbD,\HM)} \colonequals \left(\int_{\bbD}\|\MU\|_{S^p}^p\dx\dk\right)^\frac{1}{p}.
\end{align}
Similarly, $L^\infty(\bbD,V)$ refers to the Lebesgue space of essentially bounded functions.
The norm $\|\MU\|_{L^\infty(\bbD,\HM)}$ is defined accordingly.
We write $\MU\succeq 0$ for $\MU\in L^p(\bbD,\HM)$, if $\MU(x,k)\succeq 0$ for a.e. $(x,k)\in\bbD$. Any $\MU\in L^p(\bbD,\HM)$ can be written in terms of the Stokes parameters as
\begin{align*}
    \MU=\frac{1}{2}\begin{pmatrix}I + Q & U + i V\\ U - i V & I-Q \end{pmatrix}.
\end{align*}
A straightforward calculation of the singular values of $\MU$ shows that the Schatten norm of $\MU$ can then be expressed as
\begin{align}
    \|\MU\|_{S^p}^p &= \Big(\frac{1}{4}\big[I^2+Q^2+U^2+V^2]+\frac{1}{2}I\big[Q^2+U^2+V^2]^{1/2}\Big)^{p/2} \notag\\
    &+\Big(\frac{1}{4}\big[I^2+Q^2+U^2+V^2]-\frac{1}{2}I\big[Q^2+U^2+V^2]^{1/2}\Big)^{p/2}.\label{eq:sing_U}
\end{align}
This expression shows that the $L^p(\bbD,\HM)$-norm is equivalent to the $L^p(\bbD)^4$-norm on the Stokes parameters.
\subsection{Assumptions on the parameters} Let $\bbD = \RR^3\times \RR^3\backslash\{0\}$. Throughout this paper we will make the following assumptions on the parameters in Equation \cref{eq:rte_W}:

\medskip

\noindent The velocity function $\nu\in C^{1,1}(\RR^3)$ has a Lipschitz continuous and uniformly bounded gradient $\nabla_x \nu$, and there exist  constants $\nu_{\min},\nu_{\max}>0$ such that $\nu_{\min}\leq \nu(x)\leq \nu_{\max}$ for all $x\in \RR^3$.

\medskip

\noindent The functions $n:\bbD\to \RR$ and $\Sigma:\bbD\to \RR$ satisfy
\begin{align*}
     n,\Sigma\in L^\infty(\bbD),\quad\Sigma\geq 0.
\end{align*}

\medskip

\noindent The function $\Sigma(x,k)$ is related to $\sigma(x,k\cdot k')$ and $T(k,k')$ via the normalization condition
\begin{align}
\label{eq:normalization}
\Sigma(x,k) I_2 = \int_{\bbS_{|k|}} \sigma(x,k\cdot k') T(k,k') T(k,k')^*\,d\lambda(k'),
\end{align}
with $I_2$ the $2\times2$ identity matrix. Note that \eqref{eq:normalization} implies that $\Sigma(x,k)$ is radially symmetric in $k$.

\section{Radiative transfer as a system of Liouville equations}
\label{section: LiouvilleUnbounded}
In this section we consider the unbounded domain $\bbD = \bbX\times \bbK$ with $ \bbX= \RR^3, \bbK= \RR^3\backslash\{0\}$. The point $k=0$ is left out of the domain to avoid singularities in $\nabla_kH$. The case of a bounded spatial domain $\bbX_b$ is treated in \Cref{section: bounded-domain}.

We first consider the radiative transfer equation \cref{eq:rte_W} without the scattering terms $S(\MW)$ and $\Sigma\MW$ and the source term $\MF$. In this situation, \cref{eq:rte_W} can be written into a coupled system of equations for the Stokes parameters $(I,Q,U,V)$ of the coherence matrix $\MW$, see \cref{eq:Stokes}, as follows
\begin{align}
    \frac{\partial I}{\partial t} &= [I,H],      \label{eq:evo_I}\\
    \frac{\partial Q}{\partial t} &= [Q,H]+2nU,   \label{eq:evo_Q}\\
    \frac{\partial U}{\partial t} &= [U,H]-2nQ,  \label{eq:evo_U}\\
    \frac{\partial V}{\partial t} &= [V,H].  \label{eq:evo_V}
\end{align}
Here, we use the Poisson bracket $[f,g]=\nabla_k f \cdot\nabla_x g -\nabla_x f \cdot\nabla_k g $. Since the coupling only acts on $U$ and $Q$ through $n$ via a rotation, this system of equations has the following useful property. Also recall that $H(x,k)=\nu(x)|k|$.
\begin{proposition}
\label{prop:CoupledFunction}
Let the functions $I,Q,U,V$ satisfy \cref{eq:evo_I}-\cref{eq:evo_V}, then
\begin{align*}
    \frac{\partial}{\partial t}f(I,Q^2+U^2,V) = [f(I,Q^2+U^2,V),H]
\end{align*}
for any differentiable functions $f\colon\RR^3\to \RR$.
\end{proposition}

\begin{proof}
Since $I,Q,U,V$ are solutions to \cref{eq:evo_I}-\cref{eq:evo_V}, the chain rule implies that
\begin{align*}
    &\frac{\partial}{\partial t}f(I,Q^2+U^2,V) = f_1\frac{\partial I}{\partial t}+2f_2\left(Q\frac{\partial Q}{\partial t}+U\frac{\partial U}{\partial t}\right)+f_3\frac{\partial V}{\partial t}\\
    &=f_1[I,H]+2f_2\left(Q[Q,H]+U[U,H]\right)+f_3[V,H],
\end{align*}
where $f_i$ denotes the partial derivative of $f$ with respect to the $i$-th component. Since the canonical Poisson bracket contains derivatives, it satisfies the chain rule in its first argument and we have
\begin{align*}
   f_1[I,H]+2f_2(Q[Q,H]+U[U,H])+f_3[V,H] = [f(I,Q^2+U^2,V),H].
\end{align*}
\end{proof}
It can be seen that for $n=0$, Equations \cref{eq:evo_I}-\cref{eq:evo_V} decouple and the time evolution of each Stokes parameter is governed by a Liouville equation. 
For convenience of the reader and for later reference, we give a proof of the unique solvability of \cref{eq:evo_I}-\cref{eq:evo_V}.

\subsection{Method of characteristics}
We proceed to solve the Liouville equation \cref{eq:evo_I} using the method of characteristics. For this purpose we consider the characteristic equations
\begin{align}
    \frac{\partial X}{\partial t}(t) &=\hphantom{-}\nabla_k H(X(t),K(t)), \label{eq:charX}\\
    \frac{\partial K}{\partial t}(t) & =-\nabla_x H(X(t),K(t)). \label{eq:charK}
\end{align}
Note that $\nabla_k H(x,k)=\nu(x)k/|k|$ and $\nabla_x H(x,k)=|k|\nabla_x \nu(x)$. The well-posedness of the characteristic equations is proven by the following lemma.

\begin{lemma}
\label{lem:existence_XK}
For any $x_0\in\bbX$, $k_0\in\bbK$ and $T\geq0$, the characteristic equations \cref{eq:charX}--\cref{eq:charK} have a unique solution $X,K\in C^1([0,T],\RR^3)$ satisfying $X(0)=x_0$ and $K(0)=k_0$. Moreover,  it holds that
\begin{align}
    \label{eq:bound_K}
    \frac{\nu_{\min}}{\nu_{\max}}|k_0|\leq |K(t)|\leq \frac{\nu_{\max}}{\nu_{\min}}|k_0|\quad\textnormal{for } 0\leq t\leq T.
\end{align}
\end{lemma}
\begin{proof}
\textbf{Step 1.} By assumption, $\nabla_x \nu$ is Lipschitz continuous, and it follows that $\nabla_k H$ and $\nabla_x H$ are locally Lipschitz continuous for $(x,k)\in\bbD$. Hence, the right-hand side of \cref{eq:charX}--\cref{eq:charK} is locally Lipschitz in $\bbD$, and the Picard-Lindel\"of theorem, see e.g.\ \cite[Theorem 3.1]{Hale1980} ensures the existence and uniqueness of a continuously differentiable local solution $(X,K)$ that depends continuously on the initial data $(x_0,k_0)\in \bbD$.

\medskip

\textbf{Step 2.}  Observe that $H(x,k)=\nu(x)|k|$ is preserved in time, i.e., $\nu(x_0)|k_0| = H(x_0,k_0)=H(X(t),K(t))=\nu(X(t))|K(t)|$ for all $s$ for which the solution exists (from Step 1.). Thus, by the assumption that $0<\nu_{\min}\leq \nu(x)\leq \nu_{\max}<\infty$ for all $x \in \RR^3$, we obtain \cref{eq:bound_K} for all $t$ in the local existence interval. Therefore, we can extend the local solution from Step 1. to a global solution, cf. \cite[p. 18]{Hale1980}.
\end{proof}
In view of \cref{lem:existence_XK}, we can define the flow map 
\begin{equation}
\label{eq:flowmap}
\phi_t\colon (x_0,k_0)\mapsto (X(t),K(t))
\end{equation}
 for each $t>0$, and, similarly by reversing the flow, its inverse $\phi_{-t}$. In fact $\phi_{-t}(x,k)=\phi_t(x,-k)$. We note that $\phi_{t+s}=\phi_t\circ\phi_s$ for all $s,t\in\mathbb{R}$. By integration of \cref{eq:charX}--\cref{eq:charK} the flow map can be written as
\begin{align}
    \label{eq:flow_map}
    \phi_t(x_0,k_0) = \left(x_0+\int_0^t \nu(X(s))\frac{K(s)}{|K(s)|}\ddd{s},\quad k_0-\int_0^t\nabla_x \nu(X(s))|K(s)|\ddd{s}\right).
\end{align}
Because the characteristic curves are governed by a Hamiltonian system of equations, this flow map preserves phase-space volume.
\begin{lemma}[Liouville's theorem, {\cite[Theorem 8.3]{Fasano2006}}]
\label{lemma:Liouville}
For all $t\in \RR$, the flow map $\phi_t \colon\bbD\to\bbD$ defined in \eqref{eq:flowmap} preserves phase-space volume, i.e.,
\begin{align*}
    |\det D\phi_t|=1,
\end{align*}
where $D\phi_t$ denotes the Jacobian of the mapping $(x,k)\mapsto\phi_t(x,k)$.
\end{lemma}

\subsection{Well-posedness of the Liouville equation}
Using the characteristic curves, we are able to solve equations \eqref{eq:evo_I}--\eqref{eq:evo_V} if $n=0$. In this case the system of equations is completely decoupled and it suffices to consider the prototypical Liouville equation
\begin{alignat}{4}
    \frac{\partial u}{\partial t} &= [u,H] &\quad&\text{on } \bbD\times (0,T),     \label{eq:evo_poisson}\\
    u(0)&=u_0 &&\text{on } \bbD,  \label{eq:evo_poisson_IC}
\end{alignat}
with $u=u(x,k,t)$ and for some fixed $T>0$. Here, $u(0)$ denotes the function $u(\cdot,\cdot,0)$.
First we discuss how the characteristic curves relate to operator semigroups.
\begin{lemma}
\label{lemma:Semigroup}
For $1\leq p<\infty $ the flow map $\phi_t$, given in \cref{eq:flow_map}, define a strongly continuous group of contractions $\{G(t)\}_{t\in\RR}$ on $L^p(\bbD)$ via
    \begin{align}
    \label{eq:GroupUnbounded}
        (G(t)v)(x,k) \colonequals v(\phi_{-t}(x,k)), \qquad v\in L^p(\bbD).
    \end{align}
Furthermore, $G(t)$ preserves positivity; i.e., if $v$ is non-negative, then so is $G(t)v$.
\end{lemma}
\begin{proof}
Fix $v \in C_{c}(\bbD)$.  By \cref{lem:existence_XK}, for all $(x,k)\in \bbD, t\in\RR$, there exists a unique solution $X,K\in C^{1}([0,t],\mathbb{R}^{3})$ of the characteristic equations \eqref{eq:charX}--\eqref{eq:charK} with $X(t)=x, K(t)=k$.  Setting $\phi_{-t}(x,k)=(X(0),K(0))$, we deduce that
\begin{align*}
    (G(t)v)(x,k) = v(\phi_{-t}(x,k))
\end{align*}
is defined for all $t\in\RR$ and satisfies the group properties $G(0)v = v$ and $G(t+s)v = G(t)G(s)v$ for all $s,t\in\RR$. Furthermore, since $v\in C_c(\bbD)$, it holds that
\begin{align*}
    \lim_{t\xrightarrow{}0}\|(G(t)v)-v\|_{L^p(\bbD)} = \lim_{t\xrightarrow{}0}\|v(\phi_{-t}(\cdot,\cdot))-v\|_{L^p(\bbD)}=0.
\end{align*}
Liouville's theorem implies the contraction property as follows. Since $|\det D\phi_t|=1$ by \cref{lemma:Liouville}, it follows that
\begin{align*}
    \|G(t)v\|_{L^p(\bbD)} = \|v(\phi_{-t}(\cdot,\cdot))\|_{L^p(\bbD)} = \|v\|_{L^p(\bbD)}.
\end{align*}
By density of $C_c(\bbD)$ in $L^p(\bbD)$ these results carry over to functions $v\in L^p(\bbD)$. Furthermore, since the action of $G(t)$ is a translation of the function $v$, preservation of positivity is clear.
\end{proof}

Now define the space
\begin{align}\label{eq:def_Wp}
    W^p(\bbD) \colonequals \{v\in L^p(\bbD): [v,H]\in L^p(\bbD)\}.
\end{align}
Before proceeding we state some density results for $W^p(\bbD)$. We begin by arguing that neighborhoods of $k=0$ can be neglected, which uses the special form of $H$.
\begin{lemma}
\label{lemma:density_full}
The space $W^p_c(\bbD)$ of functions in $W^p(\bbD)$ with compact support in $\bbK$ is dense in $W^p(\bbD)$.
\end{lemma}
\begin{proof}
Let $B_r$ denote the open ball in $\RR^3$ of radius $r$ centred at $0$.  Let $\eta\in C^\infty_c(\RR^{3})$ be a cutoff function, i.e., $\supp\eta\subseteq B_2,\eta=1$ in $B_1$ and $0\leq\varphi\leq1$. For $n\in \mathbb{N}$, define
\begin{align*}
    u_n(x,k) \colonequals \eta_n(k)\psi_n(k)u(x,k),\quad \eta_n(k) \colonequals \eta(k/n),\quad \psi_n(k) \colonequals 1-\eta\left(nk\right).
\end{align*}
Then it holds that $\supp u_n \subseteq \Omega_n = \bbX\times (B_{2n}\backslash B_{\frac{1}{n}})$. Hence,
\begin{align*}
    u_n\to u
    \quad\text{and}\quad  [u_n,H]\to[u,H]\quad \text{almost everywhere as }n\to\infty.
\end{align*}
Since $0\leq \eta \leq 1$, we infer that $\{|u_n|^p\}_n$ is uniformly integrable. It remains to show uniform integrability of $[u_n,H]$, which, in turn, implies $u_n \to u$ in $W^p(\bbD)$ by dominated convergence. To that end, using the triangle inequality, one verifies that
\begin{align*}
    |[u_n,H]|^p\leq g_n \colonequals 3^{p-1}\left(|[u,H]|^p+|u|^p|[\eta_n,H]|^p+|u|^p|[\psi_n,H]|^p\right).
\end{align*}
Observe that $\psi_n$ is constant for $|k|\leq 1/n$ or $|k|\geq 2/n$. Hence $[\psi_n,H]$ is supported on an annulus $A_n\colonequals \{1/n \leq |k|\leq 2/n\}$. On $A_n$, we compute 
\begin{align*}
|[\psi_n,H]|=n |[\eta,H]| = n |k| |\nabla_k\eta \cdot \nabla_x \nu|\leq 2 |\nabla_k\eta \cdot \nabla_x \nu|,
\end{align*}
which is uniformly bounded by the assumptions on $\nu$.
Thus, the term $|u|^p|[\psi_n,H]|^p$ is uniformly integrable.
Since $[\eta_n,H]=[\eta,H]/n$ is uniformly bounded in $n$, we conclude that $g_n$ is uniformly integrable, which proves the assertion.
\end{proof}
\begin{lemma}
\label{lemma:density-smooth_full}
The space $C^\infty_{c}(\bbD)$ of smooth functions with compact support in $\bbD$ is dense in $W^p(\bbD)$.
\end{lemma}
\begin{proof}
Using \cref{lemma:density_full}, we can approximate any function in $v\in W^p(\bbD)$ by functions $v_n$ that vanish for $|k|<1/n$ and whose support have a smooth boundary. Given the assumed regularity on $H$, we can approximate any $v_n$ by a function in $C^\infty_c(\bbD)$ by employing \cite[Theorem~4, p.~21]{Jensen2004}, and we obtain the assertion.
\end{proof}
We now proceed with characterizing the infinitesimal generator of the group $\{G(t)\}_{t\in\RR}$.
\begin{proposition}
\label{prop:inf-generator}
 Let $1\leq p <\infty$. The infinitesimal generator $A\colon D(A)\subset L^p(\bbD)\to L^p(\bbD)$ of the strongly continuous group $\{G(t)\}_{t\in\RR}$ is given by
\begin{align*}
    A v \colonequals [v,H], \qquad v\in D(A) \colonequals W^p(\bbD).
\end{align*}
\end{proposition}
\begin{proof}
Let $v\in L^p(\bbD)$ and $\psi\in C_c^\infty(\bbD)$. Upon change of variables, we have that
\begin{align*}
    &\hphantom{=}\lim_{h\to0}\int_{\bbD}\frac{(G(h)v)(x,k)-v(x,k)}{h}\psi(x,k)\dx\dk\\
    &=\lim_{h\to0}\int_{\bbD}\frac{v(\phi_{-h}(x,k))-v(x,k)}{h}\psi(x,k)\dx\dk\\
    &=\lim_{h\to0}\int_{\bbD}v(x,k)\frac{\psi(\phi_{h}(x,k))-\psi(x,k)}{h}\dx\dk\\
    &=-\int_{\bbD}v(x,k)[\psi,H]\dx\dk.
\end{align*}
Hence, if $\lim_{h\to 0}(G(h)v-v)/h$ exists in $L^p(\bbD)$, then 
\begin{align*}
    \lim_{h\to0}\frac{1}{h}(G(h)v-v) = [v,H] \quad\text{in } L^p(\bbD).
\end{align*}
Conversely, assume that $v\in C^\infty_c(\bbD)\cap W^p(\bbD)$. Then
the identity
\[
v(\phi_{-h}(x,k)) - v(x,k) = h \int_0^1 [v,H](\phi_{-th}(x,k)) \ddd{t}
\]
implies that
\[
\|\frac{G(h)v-v}{h} - [v,H]\|_{L^p(\bbD)} \leq \|G(-th)[v,H]-[v,H]\|_{L^p(\bbD)} \to 0
\]
as $h\to 0$ by strong continuity of the group $G$, which concludes the proof using a density argument, \cref{lemma:density-smooth_full}.
\end{proof}

Using the semigroup $G(t)$ we can establish solvability of the Cauchy problem \eqref{eq:evo_poisson}--\eqref{eq:evo_poisson_IC}, and, more generally, of the inhomogeneous Cauchy problem
\begin{alignat}{4}
    \frac{\partial u}{\partial t} &= [u,H]+q &\quad&\text{on } \bbD\times (0,T), \label{eq:evo_poisson_inhom1}\\
    u(0)&=u_0 &&\text{on } \bbD. \label{eq:evo_poisson_inhom2} 
\end{alignat}
Firstly, for $q\in L^1((0,T),L^p(\bbD))$ and $u_0\in L^p(\bbD_b)$, the function $u\in C^0([0,T], L^p(\bbD))$ defined by
\begin{align}
    \label{eq:var_of_constants}
    u(t) \colonequals G(t) u_0 + \int_0^t G(t-s) q(s) \ddd{s}
\end{align}
is a mild solution of the inhomogeneous Cauchy problem \cref{eq:evo_poisson_inhom1}--\cref{eq:evo_poisson_inhom2}. 
Secondly, if $u\in C^1((0,T),L^p(\bbD))\cap C^0([0,T],W^p(\bbD))$ satisfies \cref{eq:evo_poisson_inhom1}--\cref{eq:evo_poisson_inhom2}, the $u$ is a classical solution.

The next result is a consequence of \cref{prop:inf-generator} and semigroup theory \cite[p. 106]{Pazy1983}.
\begin{corollary}
\label{cor:semigroup}
For any $T>0$, the abstract Cauchy problem \eqref{eq:evo_poisson}--\eqref{eq:evo_poisson_IC} has a unique mild solution $u\in C([0,T],L^{p}(\bbD))$ for every $u_{0}\in L^{p}(\bbD)$ given by
\begin{align*}
    u(x,k,t) = (G(t)u_0)(x,k).
\end{align*}
Furthermore if $u_{0}\in W^p(\bbD)$ then $u$ is a classical solution.
\end{corollary}
\section{Well-posedness of the radiative transfer equation on full space}
\label{section: CoupleScatter}
\Cref{cor:semigroup} shows that, for $n=0$, the uncoupled system \cref{eq:evo_I}--\cref{eq:evo_V} has a unique solution given through a group of isometries $G(t)$ with an infinitesimal generator defined in \cref{prop:inf-generator}. We now turn our attention to the full matrix-valued transport equation \cref{eq:rte_W}. 
Furthermore, via the representation of Hermitian matrices $\HM$ in terms of the Stokes parameters, see \cref{eq:Stokes}, we can identify elements of $L^p(\bbD,\HM)$ with elements of $L^p(\bbD)^4$.
%
This identification motivates the definition of the operator $\MA \colon D(\MA)\subset L^p(\bbD,\HM)\to L^p(\bbD,\HM)$,
\[\MA\MW \colonequals [\MW,H]=\begin{pmatrix} [W_{1,1},H] & [W_{1,2},H]\\ [W_{1,2}^*,H] & [W_{2,2},H]\end{pmatrix}\quad \text {for } \MW = \begin{pmatrix} W_{1,1} & W_{1,2}\\ W_{1,2}^* & W_{2,2}\end{pmatrix},\]
with domain
\begin{align}
    \label{eq:inf-generatorMatrix}
    D(\MA) \colonequals\left\{
    \MW \in L^p(\bbD,\HM) :  W_{i,j} \in W^p(\bbD) \text{ for } 1\leq i\leq j\leq 2\right\}.
\end{align}
Therefore, by extension of \cref{prop:inf-generator}, $\MA$ is the infinitesimal generator of a strongly-continuous group of contractions in $L^p(\bbD,\HM)$. To show the existence of a unique solution to \cref{eq:rte_W} we are thus left with proving that the remaining terms in \cref{eq:rte_W} can be treated as bounded perturbations of the infinitesimal generator $\MA$.

\subsection{Analysis of the coupling and scattering operators}
 We start with the analysis of the coupling operator $N(\MU)$ defined in \cref{eq:Coupling}. 
\begin{lemma}
\label{lemma:coupling}
For $1\leq p <\infty$, the coupling operator $N\colon L^p(\bbD,\HM)\to L^p(\bbD,\HM)$ defined by $N(\MU) \colonequals n(J\MU-\MU J)$ is linear and bounded with
\begin{align*}
    \|N(\MU)\|_{L^p(\bbD,\HM)} \leq 2\|n\|_{L^\infty(\bbD)}\|\MU\|_{L^p(\bbD,\HM)}.
\end{align*}
\end{lemma}
\begin{proof}
Linearity is obvious. To see boundedness, first note that
\begin{equation*}
    \|N(\MU)\|_{L^p(\bbD,\HM)} \leq \|n\|_{L^\infty(\bbD)}\|J\MU-\MU J\|_{L^p(\bbD,\HM)}.
\end{equation*} 
Since $J^*J=I_2$, the singular values of $J\MU$ and $\MU J$ coincide with those of $\MU$, whence
the assertion follows by the triangle inequality.
%
\end{proof}
For the analysis of the scattering operators $S(\MU)$ and $\Sigma$ we require the following results, which might be interesting in their own rights.
\begin{lemma}
\label{lemma:AbsoluteEstimate}
For $\MU,\MS\in \CC^{n\times n}$, with $\MS$ invertible, and for $1\leq p \leq \infty$ it holds that
\begin{align*}
    \|\MS\MU \MS\|_{S^p} \leq \|\MS|\MU|\MS\|_{S^p}.
\end{align*}
\end{lemma}
\begin{proof}
We can rewrite $\MU$ in terms of its polar decomposition \cite[p. 6]{Bhatia1997},
\begin{align*}
    \MU = \MV|\MU|,
\end{align*}
with $\MV\in \mathbb{C}^{n\times n}$ a unitary matrix. Then we can write
\begin{align*}
    \MS\MU \MS = \MS \MV|\MU|\MS = \MS \MV\MS^{-1}\MS|\MU|\MS.
\end{align*}
Using H\"older's inequality for Schatten norms \cite[p. 95]{Bhatia1997} it then holds that
\begin{align*}
    \|\MS \MV\MS^{-1}\MS|\MU|\MS\|_{S^p}\leq \|\MS \MV\MS^{-1}\|_{S^\infty}\|\MS|\MU|\MS\|_{S^p} = \|\MS|\MU|\MS\|_{S^p}.
\end{align*}
For the last inequality, we used that the matrix $\MS \MV\MS^{-1}$ can be diagonalized with eigenvalues on the unit circle in the complex plane, because $\MV$ is unitary, and that the singular values of a diagonalizable matrix equal the moduli of its eigenvalues.
%
\end{proof}
\begin{lemma}
    \label{lemma:TraceHolder}
    For Hermitian matrices $\MU,\MV\in \mathbb{C}^{n\times n}$, and an arbitrary matrix $T\in \mathbb{C}^{n\times n}$, and $1< p,q< \infty$ satisfying $ 1/p+1/q=1$, it holds that
    \begin{align*}
        \left|\Tr\left[T\MU T^*\MV\right]\right|\leq\Tr\left[T^*T|\MU|^p\right]^\frac{1}{p} \, \Tr\left[TT^*|\MV|^q\right]^\frac{1}{q}.
    \end{align*}
\end{lemma}
\begin{proof}
Consider the singular value decomposition $T=U_TS_TV_T^*$ with a real non-negative diagonal matrix $S_T$ and unitary matrices $U_T$ and $V_T$. Then, since the trace is invariant under cyclic permutations, and by using Hölder's inequality for Schatten norms it holds that
\begin{align*}
    \left|\Tr\left[T\MU T^*\MV\right]\right|=\left|\Tr\left[U_T S_TV_T^* \MU V_T S_T U_T^* \MV\right]\right|=\left|\Tr\left[S_TV_T^* \MU V_T S_T U_T^* \MV U_T\right]\right|\\
=\left|\Tr\left[ \left(S^{1/p}_T\widetilde \MU S^{1/p}_T\right)\left(S^{1/q}_T \widetilde\MV S^{1/q}_T\right)\right]\right|\leq \big\|S^{1/p}_T\widetilde \MU S^{1/p}_T\big\|_{S^p} \big\|S^{1/q}_T \widetilde\MV S^{1/q}_T\big\|_{S^q},
\end{align*}
with $\widetilde \MU=V_T^* \MU V_T$ and $\widetilde\MV=U_T^* \MV U_T$.
By applying \cref{lemma:AbsoluteEstimate}, and subsequently using the Lieb-Thirring inequality \cite{Lieb1976}
we obtain
\begin{align*}
    &\big\|S^{1/p}_T\widetilde \MU S^{1/p}_T\big\|_{S^p}= \Tr\left[ \left|S^{1/p}_T\widetilde \MU S^{1/p}_T\right|^p\right]^\frac{1}{p}\leq\Tr\left[\left(S_T^{1/p}|\widetilde{\MU}|S_T^{1/p}\right)^p\right]^\frac{1}{p}
    \leq\Tr\left[ S_T|\widetilde \MU|^p S_T\right]^\frac{1}{p}.
\end{align*}
Since $V_T$ is unitary, it holds that
\begin{align*}
    |\widetilde{\MU}|^2 = V_T^*|\MU|^2V_T = \left(V_T^*|\MU|V_T\right)^2,
\end{align*}
i.e., $|\widetilde \MU| = V_T^*|\MU|V_T$. 
By the spectral theorem it then holds that $|\widetilde{\MU}|^p = V_T^*|\MU|^pV_T$ and we can write
\begin{align*}
    \Tr\left[ S_T|\widetilde \MU|^p S_T\right]^\frac{1}{p}= \Tr\left[ S_TV_T^*|\MU|^pV_T S_T\right]^\frac{1}{p} =\Tr\left[T^*T|\MU|^p\right]^\frac{1}{p}.
\end{align*}
The term containing $\MV$ follows similarly, from which we obtain the required result.
\end{proof}
The $p=1$ and $p=\infty$ cases have to be treated separately as is done in the following
\begin{lemma}
\label{lemma: TraceInequality}
 For Hermitian matrices $\MU,\MV\in \mathbb{C}^{n\times n}$, and an arbitrary matrix $T\in \mathbb{C}^{n\times n}$, it holds that
\begin{align*}
    &\left|\Tr\left[T\MU T^*\MV\right]\right|\leq\Tr\left[T^*T|\MU|\right]\cdot\sigma_{1}(\MV),\\
    &\left|\Tr\left[T\MU T^*\MV\right]\right|\leq\sigma_{1}(\MU)\cdot\Tr\left[TT^*|\MV|\right].
\end{align*}
\end{lemma}
\begin{proof}
As in the proof of \cref{lemma:TraceHolder} we use the the singular value decomposition of $T$ and Hölders inequality for Schatten norm to obtain
\begin{align*}
    \left|\Tr\left[T\MU T^*\MV\right]\right|&\leq\|S_T\widetilde \MU S_T\|_{S^1} \|\widetilde\MV \|_{S^\infty}\\
    &=\Tr\left[S_TV_T^*|\MU|V_TS_T\right] \sigma_{1}(U_T^*\MV U_T)\\
    &=\Tr\left[T^*T|\MU|\right] \sigma_{1}(\MV).
\end{align*}
The second inequality follows in a similar fashion.
\end{proof}
Using \Cref{lemma:TraceHolder} and \Cref{lemma: TraceInequality}, we obtain the following key inequality for the scattering operator.
\begin{proposition}
    \label{proposition: ScatterBound}
    Let $\MU\in L^p(\bbD,\HM), \MV\in L^q(\bbD,\HM), 1 \leq p,q\leq \infty$, such that $1/p+1/q=1$. Then the following inequality holds,
    \begin{align*}
        \int_{\bbD}\left|\Tr\left[S(\MU)\MV\right]\right| \dx\dk\leq \|\Sigma^{1/p}\MU\|_{L^p(\bbD,\HM)} \|\Sigma^{1/q}\MV\|_{L^q(\bbD,\HM)}.
    \end{align*}
\end{proposition}
\begin{proof}
Let $1<p<\infty$.
By definition of $S$, we have that
\begin{align*}
    &\int_\bbD\left|\Tr\left[S(\MU)\MV\right]\right|\dx\dk \\
    &= \int_\bbX\int_\bbK\Big|\Tr\Big[\int_{\bbS_{|k|}}\sigma(x,k\cdot k')T(k,k')\MU(x,k')T(k',k)\MV(x,k)\ddd{k'}\Big]\Big|\dx\dk
\end{align*}
Taking the trace into the inner integral and using \cref{lemma:TraceHolder}, i.e.,
\begin{align*}
   \Tr\Big[T(k,k')\MU(x,k')T(k',k)\MV(x,k)\Big]\Big|\leq  &\Tr\Big[T(k',k)T(k,k')|\MU|^p(x,k')\Big]^\frac{1}{p}\\
    &\Tr\Big[T(k,k')T(k',k)|\MV|^q(x,k)\Big]^\frac{1}{q}.
\end{align*}
H\"older's inequality and Fubini's theorem yield that
\begin{align*}
    \int_\bbD|\Tr\big[S(\MU)\MV\big]|&\dx\dk \leq\\
    &\Bigg(\int_\bbX\int_\bbK\Tr\Big[\int_{\bbS_{|k|}}\sigma(x,k\cdot k')T(k',k)T(k,k')\dk |\MU|^p(x,k')\Big] \ddd{k'} \dx \Bigg)^\frac{1}{p} \\
    &\Bigg(\int_\bbX\int_\bbK\Tr\Big[\int_{\bbS_{|k|}}\sigma(x,k\cdot k')T(k,k')T(k',k)\ddd{k'}|\MV|^q(x,k)\Big]\dk\dx\Bigg)^\frac{1}{q}.
\end{align*}
In view of the normalization condition \cref{eq:normalization}, we further deduce that
\begin{align*}
    &\int_\bbD|\Tr\big[S(\MU)\MV\big]|\dx\dk \leq\\
    &\Bigg(\int_\bbX\int_\bbK\Tr\Big[\Sigma(x,k')|\MU|^p(x,k')\Big]\ddd{k'}\dx\Bigg)^\frac{1}{p} \Bigg(\int_\bbX\int_\bbK\Tr\Big[\Sigma(x,k)|\MV|^q(x,k)\Big]\dk\dx\Bigg)^\frac{1}{q}\\
    &=\|\Sigma^{1/p}\MU\|_{L^p(\bbD,\HM)}\cdot\|\Sigma^{1/q}\MV\|_{L^q(\bbD,\HM)}.
\end{align*}
For the case $p=1$, we use \cref{lemma: TraceInequality} to obtain the estimate
\begin{align*}
    &\int_\bbD|\Tr\big[S(\MU)\MV\big]|\dx\dk \leq\\
    &\int_\bbX\int_\bbK\int_{\bbS_{|k|}}\sigma(x,k\cdot k')\Tr\Big[T(k',k)T(k,k')|\MU|(x,k')\Big]\, \sigma_{1}(\MV(x,k))\ddd{k'}\dx\dk.
\end{align*}
By then applying H\"older, Fubini and the normalization condition \cref{eq:normalization} just as for the case $p>1$ we obtain
\begin{align*}
    \int_\bbD|\Tr\big[S(\MU)\MV\big]|\dx\dk \leq \|\Sigma\MU\|_{L^1(\bbD,\HM)} \|\MV\|_{L^\infty(\bbD,\HM)}.
\end{align*}
The case $p=\infty$ follows by interchanging the roles of $\MU$ and $\MV$ from the case $p=1$.
\end{proof}
Combining the previous result with a duality argument shows that the scattering operator is bounded.
\begin{proposition}
\label{proposition:Scattering}
For $\MU \in L^p(\bbD,\HM), 1\leq p\leq\infty$, it holds that
\begin{align*}
    \|S(\MU)\|_{L^p(\bbD,\HM)}\leq\|\Sigma\|_{L^\infty(\bbD)}\|\MU\|_{L^p(\bbD,\HM)}.
\end{align*}
\end{proposition}
\begin{proof}
Let $q$ be such that $1/p+{1}/{q}=1$. By duality, it holds that
\begin{align*}
    \|S(\MU)\|_{L^p(\bbD,\HM)} = \sup\int_{\bbD}|\Tr\left[S(\MU)\MV\right]|\dk\dx,
\end{align*}
where the supremum is taken over all $\MV\in L^q(\bbD,\HM)$ satisfying $\|\MV\|_{L^q(\bbD,\HM)} = 1$.
By \cref{proposition: ScatterBound}, it then follows for $1<p<\infty$ that
\begin{align*}
    \|S(\MU)\|_{L^p(\bbD,\HM)} \leq{}& \sup_{\|\MV\|\in L^q(\bbD,\HM) = 1}\|\Sigma^{1/p}\MU\|_{L^p(\bbD,\HM)}\|\Sigma^{1/q}\MV\|_{L^q(\bbD,\HM)}\\
    \leq{}& \|\Sigma\|_{L^\infty(\bbD)}\|\MU\|_{L^p(\bbD,\HM)}.
\end{align*}
The cases $p=1$ and $p=\infty$ follow with slight modifications.
\end{proof}
\subsection{Well-posedness of the radiative transfer equation}
According to the previous section, the scattering and coupling operators are bounded linear operators on $L^p(\bbD,\HM)$. Thus, the sum of these operators and $\MA$, defined in \cref{eq:inf-generatorMatrix}, generate a semigroup \cite[p. 76]{Pazy1983}, which, in turn, establishes well-posedness of \cref{eq:rte_W}.
In order to show that this semigroup acts on the cone of positive functions in $L^p(\bbD,\HM)$ we require the following result before stating an existence result.
\begin{lemma}
\label{lemma:ScatterPositivity}
The scattering operator $S \colon L^p(\bbD,\HM)\to L^p(\bbD,\HM)$ defined in \cref{eq:Scattering} preserves positivity; i.e., $S(\MU)\succeq 0$ for all $\MU\in L^p(\bbD,\HM)$ with $\MU \succeq 0$.
\end{lemma}
\begin{proof}
Let $\MU\in L^p(\bbD,\HM)$ such that $\MU\succeq 0$ and define
\begin{align*}
    \widetilde{\MU}(x,k,k') \colonequals \sigma(x,k\cdot k')T(k,k')\MU(x,k') T^*(k,k'),
\end{align*}
with $\sigma$ and $T$ as given in \cref{eq:Scattering}. Since $T^*(k,k')=T(k,k')$ and $\sigma\geq 0$, we have that $\widetilde{\MU}(x,k,k')\succeq 0$. The assertion $S(\MU)\succeq 0$ then follows from
\begin{align*}
    \zeta^*S(\MU)(x,k)\zeta = \int_{\bbS_{|k|}} \zeta^* \widetilde{\MU}(x,k,k')\zeta \ddd{\lambda(k')}\geq 0\quad\text{for any } \zeta\in\mathbb{C}^2.
\end{align*}
\end{proof}
\begin{theorem}
\label{thrm: ExistenceFull}
Let $1\leq p <\infty$, and let $T>0$, $\MW_0\in L^p(\bbD,\HM)$, and $\MF\in C([0,T],L^p(\bbD,\HM))$, then the radiative transfer equation \cref{eq:rte_W} has a unique mild solution $\MW\in C([0,T],L^p(\bbD,\HM))$, satisfying
\begin{align*}
    \MW(0) =\MW_0.
\end{align*}
Furthermore, if $\MW_0$ and $\MF$ are positive, then so is $\MW$.
If $\MW_0\in D(\MA)$ and $\MF\in C^1([0,T],L^p(\bbD,\HM))$ then $\MW$ is a classical solution.
\end{theorem}
\begin{proof}
\textbf{Existence.}
In view of \cref{lemma:coupling} and \cref{proposition:Scattering} the operators $N,S$ and $\Sigma$ are bounded linear operators on $L^p(\bbD,\HM)$. Using semigroup theory, cf.\cite[p. 76]{Pazy1983} or \cite[p. 348]{DL5}, the perturbed operator $\MA+N+S-\Sigma$ with domain $D(\MA)$ generates a strongly continuous semigroup $L(t)$.
Thus, there exists a unique mild solution of \cref{eq:rte_W} for every initial condition $\MW_0\in L^p(\bbD,\HM)$, which becomes a classical solution if $\MW_0\in D(\MA)$ and $\MF\in C^1([0,T],L^p(\bbD,\HM))$ \cite[pp. 106-107]{Pazy1983}.

\textbf{Positivity.}
Trotters product formula, see, e.g., \cite[p. 464]{DL5}, implies that
\begin{align*}
    L(t) \MU= \lim_{m\to\infty}\left(e^{\frac{t}{m}\MA}e^{\frac{t}{m}N}e^{\frac{t}{m}S}e^{-\frac{t}{m}\Sigma}\right)^m \MU
\end{align*}
for all $\MU\in L^p(\bbD,\HM)$.
To show positivity of $L(t)$, it thus suffices to show positivity of $e^{t\MA},e^{tN},e^{tS}$ and $e^{-t\Sigma}$ separately. Positivity of $e^{t\MA}$ follows from \cref{lemma:Semigroup}. Furthermore, since $\Sigma$ is a function, positivity of $e^{-t\Sigma}$ follows from the identity
\begin{align*}
    \left(e^{-t\Sigma}\MU\right)(x,k) = e^{-t\Sigma(x,k)}\MU(x,k).
\end{align*}
By \cref{lemma:ScatterPositivity}, $S(\MU)\succeq 0$ if $\MU\succeq 0$, and positivity of $e^{tS}$ then follows from the exponential formula
\begin{align*}
    e^{tS}\MU = \sum_{m=0}^\infty\frac{t^m}{m!}S^m\MU.
\end{align*}
By decomposition of elements in $L^p(\bbD,\HM)$ into the Stokes parameters, the coupling operator $N$ can be identified with the mapping
\begin{align*}
    \begin{pmatrix} I \\ Q\\U\\V
    \end{pmatrix}
    \mapsto
    \begin{pmatrix}
    0 & 0 & 0 & 0\\
    0 & 0 & 2n & 0\\
    0 & -2n & 0 & 0\\
    0 & 0 & 0 &0
    \end{pmatrix}
    \begin{pmatrix} I \\ Q\\U\\V
    \end{pmatrix}.
\end{align*}
Denoting $c(t)\colonequals \cos(2 t n)$ and $s(t)\colonequals \sin( 2 t n)$,
the exponential of $N$ is then identified by the rotation
\begin{align*}
    \begin{pmatrix} I \\ Q\\U\\V
    \end{pmatrix}
    \mapsto
    \begin{pmatrix}
    1 & 0 & 0 & 0\\
    0 & c(t) & s(t) & 0\\
    0 & -s(t) & c(t) & 0\\
    0 & 0 & 0 &1
    \end{pmatrix}
    \begin{pmatrix} I \\ Q\\U\\V
    \end{pmatrix},
\end{align*}
or, equivalently, by rewriting the Stokes parameters into a matrix,
\begin{align*}
    e^{tN}: \begin{pmatrix}
    I+Q & U+iV\\
    U-iV & I-Q
    \end{pmatrix}
    \mapsto
    \begin{pmatrix}
    I+c(t)Q+s(t)U & c(t)U-s(t)Q+iV\\
    c(t)U-s(t)Q-iV & I-c(t)Q-s(t)U
    \end{pmatrix}.
\end{align*}
From this representation, we observe that $e^{tN}$ preserves the eigenvalues and thus positivity. From Trotters formula it can then be concluded that the semigroup $L(t)$ operates in the cone of positive function in $L^p(\bbD,\HM)$.
The assertion then follows from the variations of constants formula, cf. \cref{eq:var_of_constants}.
\end{proof}



\section{Liouville equation on a bounded domain}
\label{section: bounded-domain}
We consider the existence and uniqueness of solutions to Liouville's equation on domains $\bbD_b = \bbX_b\times\bbK$ with a bounded spatial component $\bbX_b$, which is the counterpart of \cref{eq:evo_poisson}--\cref{eq:evo_poisson_IC}.
Let us fix some notation.
We assume that $\partial \bbX_b$ is of class $C^1$.
Since $\bbK = \RR^3\backslash\{0\}$, the boundary of $\bbD_b$ is given by $\partial\bbD_b = (\partial \bbX_b\times\bbK)\cup(\bbX_b\times \{0\})$.
Since $(\bbX_b\times \{0\})$ is a set of measure zero, we will from this point on write, in slight abuse of notation, $\partial\bbD_b \colonequals \partial \bbX_b\times\bbK$.
Denoting by $\vec{n}$ the continuous unit outward normal on $\bbX_b$, we then define the in- and outflow boundary, denoted by $\Gamma_-$ and $\Gamma_+$, respectively, as
\begin{align}
    \label{eq:InflowBoundary}
    \Gamma_\pm \colonequals \{(x,k)\in \partial\bbD_b \colon \pm\nabla_kH(x,k)\cdot\vec{n}(x)>0\}.
\end{align}
Since $\nabla_k H(x,k)=\nu(x) k/|k|$ and $\nu(x)\geq \nu_{\min}>0$, we note that the tangential part of the boundary,
\begin{align*}
    \Gamma_0 \colonequals \{(x,k)\in \partial\bbD_b \colon \nabla_k H(x,k)\cdot\vec{n}(x)=0\},
\end{align*}
has measure zero.


Having fixed the notation, we study the following following boundary value problem in the forthcoming subsections,
\begin{alignat}{6}
   \frac{\partial u}{\partial t}+[H,u] &= q&\quad& \text{on }\bbD_b\times(0,T], \label{prob:LiouvilleBoundNonHom1}\\
     u &=  g& &\text{on }\Gamma_-\times(0,T], \label{prob:LiouvilleBoundNonHom2}\\
    u(0) &= u_0 &&\text{on }\bbD_b. \label{prob:LiouvilleBoundNonHom3}
\end{alignat}
Here $T>0$ is fixed, $g$ is an inflow boundary condition and $q$ is an internal source term. 
We will first consider homogeneous boundary conditions $g=0$, and later deal with the case $g\neq 0$ by using a suitable extension.

\subsection{Auxiliary tools}
We define for any point $(x,k)\in\bbD_b$ the travel times $\tau_-(x,k)$ and $\tau_+(x,k)$ as 
\begin{align}\label{eq:traveltime}
&\tau_\pm(x,k) \colonequals \sup\{t>0: \, \phi_{\pm s}(x,k)\in \bbD_b,\ 0\leq s<t\},
\end{align}
with $\phi_t$ the flow map as defined in \cref{eq:flow_map}. This generalizes the usual notion of travel times in \cite[p.~221]{DL6}. The travel time $\tau_-(x,k)$ is the time a particle at $(x,k)$ needs to reach the inflow boundary $\Gamma_-$ along the flow $\phi_{-t}$, and $\tau_+(x,k)$ denotes the time a particle at $(x,k)$ needs to reach the outflow boundary $\Gamma_+$ along the flow $\phi_{t}$.

The characteristic curves defined by the flow map $\phi_t$ can be used to parametrize integrals over $\bbD_b$ by integration over the inflow boundary. 
We do not aim to prove that the characteristics cover $\bbD_b$, but consider
\begin{align}\label{def:U}
\bbU \colonequals \{(x,k)\in \bbD_b \colon (x,k) = \phi_t(x_0,k_0),\  (x_0,k_0)\in\Gamma_-, 0\leq t\leq \tau_+(x_0,k_0)\}.
\end{align}
Using a change of variables and that phase-space volume is preserved by $\phi_t$, we prove the next statement, cf. \cite[Lemma~2.1]{Stefanov:1999} for the case of constant $\nu$, for which $\bbU=\bbD_b$.
\begin{proposition}
\label{prop:TransformU}
Let $\bbU$ be defined in \cref{def:U}. Then, for all integrable functions $f\colon \bbU\to\RR$ it holds that
\begin{align}
    \label{eq:BoundaryTransform}
    \int_\bbU f(x,k)\dx\dk = \int_{\Gamma_-}\int_0^{\tau_+(x_0,k_0)} f(\phi_t(x_0,k_0))|\nabla_k H(x_0,k_0)\cdot \vec{n}(x_0)|\dt\ddd{\Gamma}.
\end{align}
\end{proposition}
\begin{proof}
Every point $(x,k)\in\bbU$ can be expressed as $(x,k)=\psi(t,x_0,k_0)=\phi_t(x_0,k_0)$ with $(x_0,k_0)\in \Gamma_-$ and $0<t<\tau_+(x_0,k_0)$. Using that $\partial\bbX_b$ is of class $C^1$, and using \cref{eq:flow_map}, we can parametrize $\partial\bbX_b$ such that the derivative of $\psi$ with respect to $x_0$ along $\Gamma_-$ satisfies
\begin{align*}
    (D_{x_0}\psi)_{\mid t=0} = \begin{pmatrix} \vec{\tau}_1^* & {\bf 0} \\ \vec{\tau}_2^* & {\bf 0}\end{pmatrix},
\end{align*}
with $\vec{\tau}_1$, $\vec{\tau}_2$ denoting the unit tangent vector field at $\partial\bbX_b$ oriented such that $\vec{\tau}_1 \times \vec{\tau}_2=\vec{n}$, and ${\bf 0}$ a zero row of length three.
Denoting $D\psi$ the Jacobian of $\psi$, a change of variables implies that
\[\int_\bbU f(x,k)\dx\dk = \int_{\Gamma_-}\int_0^{\tau_+(x_0,k_0)} f(\phi_t(x_0,k_0))|\det(D\psi)|\dt\ddd{\Gamma}. \]
To evaluate $\det(D\psi)$, we first compute $\det(D\psi)$ for $t=0$. 
Using \cref{eq:flow_map}, we obtain that
\[\det (D\psi)|_{t=0} = \det \begin{pmatrix}
(\nabla_k H)^* & -(\nabla_xH)^*\\
\vec{\tau}_1^*& {\bf 0}\\
\vec{\tau}_2^*& {\bf 0}\\
{\bf 0}_3 & I_3
\end{pmatrix}= \det
\begin{pmatrix}
\nabla_k H & \vec{\tau}_1 & \vec{\tau}_2 
\end{pmatrix}=\nabla_k H \cdot (\vec{\tau}_1\times \vec{\tau}_2).\] 
To compute the determinant for $t>0$, we rewrite the characteristic equations \cref{eq:charX}-\cref{eq:charK} in terms of the flow map,
\[\frac{d}{dt}\phi_t(x_0,k_0) = F(\phi_t(x_0,k_0)),\]
with divergence free vector field $F(x,k) = (\nabla_kH(x,k),-\nabla_xH(x,k))$. It then follows from the chain rule that
\[\frac{d}{dt}D\psi(t,x_0,k_0) =  DF(\psi(t,x_0,k_0)) D\psi(t,x_0,k_0).\]
Because $\Tr(DF)=\divergence(F)=0$, Liouville's formula \cite[Lemma 3.11]{Teschl2004} implies that
\[\frac{d}{dt}\det(D\psi) = \divergence(F) D\psi = 0,\]
which concludes the proof.
\end{proof}
The next result asserts that almost all characteristic curves that cross the inflow boundary are of finite length, i.e., cross the outflow boundary.
\begin{lemma}
\label{lem:TravelTime}
The set $\Gamma_\infty \colonequals \{(x_0,k_0) \in \Gamma_-: \tau_+(x_0,k_0)=\infty\}$ has measure zero.
\end{lemma}
\begin{proof}
Suppose the set $\Gamma_\infty$ has nonzero measure. Then there exists a compact subset $\Gamma'\subset\Gamma_\infty$ with finite positive measure.
Now we define the set $\bbU'\subset \bbU$ as
\begin{align*}
    \bbU' \colonequals \{(x,k)\in \bbU: (x,k) = \phi_t(x_0,k_0), (x_0,k_0)\in \Gamma', 0\leq t\leq \infty\}.
\end{align*}
By compactness of $\Gamma'$ there exist constants $k_{min},k_{max}>0$ such that $k_{min}\leq |k|\leq k_{max}$ for all $k\in \Gamma'$. By boundedness of $\bbX_b$ and by the bounds on the $k$ variable from the flow in \cref{eq:bound_K} we deduce that $\bbU'$ has finite measure.

Since $|\nabla_kH\cdot \vec{n}|$ is continuous and strictly positive on $\Gamma_-$, there exists $\varepsilon>0$ such that $|\nabla_kH\cdot \vec{n}|\geq\varepsilon$ on $\Gamma'$. Then we have by \cref{eq:BoundaryTransform}

\[{\rm meas}(\bbU')=\int_{\Gamma'}\int_0^\infty|\nabla_kH\cdot \vec{n}|\dt  \ddd{\Gamma}\geq\varepsilon\int_{\Gamma'}\int_0^\infty \dt  \ddd{\Gamma}=\infty.\]
This leads to a contradiction and the set $\Gamma_\infty$ thus has measure zero.
\end{proof}

\cref{prop:TransformU} and \cref{lem:TravelTime} can also be stated for the outflow boundary $\Gamma_+$ by considering the inverse flow map $\phi_{-t}$.
\subsection{Traces}
To treat nonhomogeneous boundary conditions in \cref{prob:LiouvilleBoundNonHom1}--\cref{prob:LiouvilleBoundNonHom2}, we introduce the space
\begin{align*}
    \bbT^p_- \colonequals L^p(\Gamma_-;\tau_\kappa|\nabla_kH\cdot\vec{n}|\ddd{\sigma} \dk), \quad\tau_\kappa \colonequals \min(\tau_+,\kappa), 
\end{align*}
for some $\kappa>0$ fixed.
\begin{remark}\label{rem:traces}
The norms of the trace space $\bbT_\pm^p$ depend on $\kappa$, but are equivalent for different choices of this parameter. In fact, let $0<\kappa_1 \leq \kappa_2$. Then one can show that for integrable functions $f:\Gamma_-\to \RR$
\[
\frac{\kappa_1}{\kappa_2} \int_{\Gamma_-} |f| \tau_{\kappa_2} \ddd{\Gamma} \leq  \int_{\Gamma_-} |f| \tau_{\kappa_1} \ddd{\Gamma} \leq    \int_{\Gamma_-} |f| \tau_{\kappa_2} \ddd{\Gamma}.
\]
\end{remark}
We will denote by $\|f\|_{\bbT_-^p}$ the norm induced by $\kappa = 1$. It can be shown that every function in $W^p(\bbD_b)$ has a trace in $\bbT^p_-$. To show this, we require the following density results, which are proven as in \cref{lemma:density_full} and \cref{lemma:density-smooth_full}.
\begin{lemma}
\label{lemma:density}
The space $W^p_c(\bbD_b)$ of functions in $W^p(\bbD_b)$ with compact support in $\bbK$ is dense in $W^p(\bbD_b)$.
\end{lemma}
\begin{lemma}
\label{lemma:density-smooth}
The space $C^\infty_{c}(\overline{\bbX_b}\times \bbK)$ of functions that are smooth up to the boundary of $\bbX_b$ and that have compact support in $\bbK$ is dense in $W^p(\bbD_b)$.
\end{lemma}
Using density of smooth functions and adapting the proof of \cite[Theorem~ 2.2]{Manteuffel1999} to our situation, one can prove the following trace lemma.
\begin{lemma}
\label{lemma:TraceTheorem}
There exists a surjective, bounded linear operator $\gamma_- \colon W^p(\bbD_b)\to\bbT^p_-$ that satisfies $\gamma_-u=u_{\mid\Gamma_-}$ for functions $u\in W^p(\bbD_b)$ that are continuous up to the boundary of $\bbX_b$.
\end{lemma}
\begin{proof}
By \cref{lemma:density-smooth} it suffices to consider functions $v\in C^\infty(\overline{\bbX_b}\times\bbK)$.
For an arbitrary point $(x,k)\in\Gamma_-$, let $(\Tilde{x},\Tilde{k}) = \phi_{\Tilde{s}}(x,k)$ denote a point where $|v(\phi_s(x,k))|$ is minimal for $0\leq s\leq \tau_1(x,k)$.
Then we have
\begin{align}\label{eq:tr1}
    |v(\Tilde{x},\Tilde{k})|^p\tau_1\leq\int_0^{\tau_1(x,k)}|v(\phi_s(x,k))|^p \ddd{s}\leq\int_0^{\tau_+(x,k)}|v(\phi_s(x,k))|^p \ddd{s}.
\end{align}
By applying the fundamental theorem of calculus, and noting that the derivative of $v(\phi_s(x,k))$ with respect to $s$ equals $[v,H]$ evaluated in $\phi_s(x,k)$, it also holds that
\begin{align*}
    |v(x,k)|^p\leq|v(\Tilde{x},\Tilde{k})|^p
    +p\int_0^{\tau_+(x,k)}\big|[v,H] v(\phi_s(x,k))^{p-1}\big|\ddd{s}.
\end{align*}
Multiplying this inequality by $\tau_1$, integrating over $\Gamma_-$,
and using \cref{eq:tr1}, we obtain that
\begin{align*}
    \|v\|_{\bbT_-^p}^p &= \int_{\Gamma_-}|v(x,k)|^p\tau_1|\nabla_k H\cdot \vec{n}(x)|\ddd{\Gamma}\\
    \nonumber
    &\leq \int_{\Gamma_-}|v(\tilde{x},\tilde{k})|^p\tau_1|\nabla_k H\cdot \vec{n}(x)|\ddd{\Gamma}\\
    \nonumber
    &\quad+p\int_{\Gamma_-}\int_0^{\tau_+(x,k)}\big|[v,H]v(\phi_s(x,k))^{p-1}\big|\tau_1|\nabla_k H\cdot \vec{n}(x)|\ddd{s}\ddd{\Gamma}.
\end{align*}
By applying Young's inequality to the last term and using that $\tau_1\leq 1$, we obtain
\begin{align*}
    \|v\|_{\bbT_-^p}^p 
    &\leq p\|v\|_{L^p(\bbD_b)}^p+\|v\|_{W^p(\bbD_b)}^p,
\end{align*}
which shows existence of a bounded trace operator $\gamma_-$. To show that $\gamma_-$ is surjective, consider a function $g\in \bbT_-^p$ and define its extension into $W^p(\bbD_b)$ as
\begin{align}\label{eq:lifting}
\ME g(\phi_s(x,k)) =e^{-s}g(x,k), \qquad 0\leq s\leq\tau_+(x,k),
\end{align}
which defines $\ME g$ on $\bbU$, defined in \cref{def:U}. On the complement of $\bbU$, we set $\ME g=0$.
Differentiation with respect to $s$ shows that $[\ME g,H]=-\ME g$.
It remains to estimate the $L^p(\bbD_b)$-norm of $\ME g$ in terms of the $\bbT_-^p$-norm of $g$.
Using \cref{eq:BoundaryTransform}, we obtain that
\begin{align*}
    \|\ME g\|_{L^p(\bbD_b)}^p&= \int_{\Gamma_-}\int_0^{\tau_+(x,k)} e^{-ps} \ddd{s}|g(x,k)|^p|\nabla_kH\cdot\vec{n}|\ddd{\Gamma}.
\end{align*}
Since $1-e^{-t}\leq t$ for all $t\geq 0$, we deduce that
\begin{align*}
    \int_0^{\tau_+(x,k)} e^{-ps} \ddd{s} = \frac{1}{p} (1- e^{-p\tau_-(x,k)})\leq \tau_1(x,k),
\end{align*}
which concludes the proof.
\end{proof}

In some situations we require functions that have traces with more regularity than merely $\bbT_-^p$. For this purpose we define the spaces
\begin{align*}
  \widetilde \bbT^p_\pm &\colonequals L^p(\Gamma_\pm;|\nabla_k H\cdot\vec{n}|\ddd{\Gamma})\\
    \widetilde{W}^p(\bbD_b) &\colonequals \{u\in W^p(\bbD_b): u|_{\Gamma_\pm} \in \widetilde \bbT^p_\pm\},
\end{align*}
which we endow with the norms
\begin{align*}
 \|w\|^p_{\widetilde \bbT^p_\pm}  \colonequals \int_{\Gamma_\pm} |w|^p |\nabla_k H\cdot \vec{n}|\ddd{\Gamma},\qquad
 \|w\|^p_{\widetilde{W}^p(\bbD_b)}  \colonequals \|w\|_{W^p(\bbD_b)}^p + \|w\|^p_{\widetilde \bbT^p_\pm}.
\end{align*}
For $w\in\widetilde W^1(\bbD_b)$ we have the following divergence theorem
\begin{align}\label{eq:divergence}
    \int_{\bbD_b}[w,H]\dx\dk = \int_{\Gamma_+}w|\nabla_kH\cdot\vec{n}|\ddd{\Gamma}-\int_{\Gamma_-}w|\nabla_kH\cdot\vec{n}|\ddd{\Gamma},
\end{align}
which follows from \cref{lemma:density-smooth}.
It can be shown that functions in this space only require this additional regularity on one part of the domain boundary. See also \cite[p. 224]{DL6} for the case of constant $\nu$.
\begin{proposition}
\label{prop:BoundaryWithoutTau}
It holds that
\begin{align*}
    \widetilde{W}^p(\bbD_b) = \widetilde{W}^p_-(\bbD_b) \colonequals \{u\in W^p(\bbD_b): u|_{\Gamma_-} \in L^p(\Gamma_-;|\nabla_kH\cdot\vec{n}|\ddd{\Gamma})\}.
\end{align*}
\end{proposition}
\begin{proof}
Let $u\in \widetilde{W}^p_-(\bbD_b)\cap C^\infty_c(\overline{\bbX_b}\times\bbK)$. Since $u\in W^p(\bbD_b)$, the integral
\begin{align*}
    \int_{\bbD_b}[|u|^p,H]\dx\dk=p\int_{\bbD_b}u|u|^{p-2}[u,H]\dx\dk
\end{align*}
is well-defined by H\"older's inequality.
By \cref{eq:divergence} it holds that
\begin{align}\label{eq:int_parts_p}
    \int_{\bbD_b}[|u|^p,H]\dx\dk = \int_{\Gamma_+}|u|^p|\nabla_kH\cdot\vec{n}|\ddd{\Gamma}-\int_{\Gamma_-}|u|^p|\nabla_kH\cdot\vec{n}|\ddd{\Gamma}.
\end{align}
Using a density argument and that the integrals over $\Gamma_-$ and $\bbD_b$ are well-defined by assumption, we conclude that the integral over $\Gamma_+$ exists for $u\in  \widetilde{W}^p_-(\bbD_b)$.
\end{proof}

\subsection{Well-posedness of the Liouville equation on bounded domains}
We can now establish the existence of a unique solution to the Liouville equation on $\bbD_b$. 
We start by suitably adapting the definition of the group in \cref{eq:GroupUnbounded} employed for unbounded domain $\bbX$.
Since for times exceeding the travel time $\tau_-$ the characteristic curves trace back to points outside $\bbD_b$, we now define
\begin{align}
    \label{eq:semigroupbounded}
    (G(t)v)(x,k) \colonequals v(\phi_{-t}(x,k))Y(\tau_-(x,k)-t).
\end{align}
for functions $v \in L^p(\bbD_b)$ extended by zero on the complement of $\bbD_b$, and with $Y$ denoting the Heaviside step function, with $Y(0)=1/2$. 
It can be shown that this family of operators defines a semigroup on $L^p(\bbD_b)$ as follows.
\begin{proposition}
\label{prop:semigroupbounded}
The family of operators $\{G(t)\}_{t\geq0}$ defined in \cref{eq:semigroupbounded} constitutes a strongly continuous semigroup of contractions in $L^p(\bbD_b)$, for $1\leq p<\infty$. Furthermore, the infinitesimal generator $A$ of this semigroup is defined by
\begin{align*}
    Av  &\colonequals [v,H],\\
    D(A) &\colonequals \{v \in W^p(\bbD_b): v_{\mid\Gamma_{-}} = 0\}.
\end{align*}
\end{proposition}
\begin{proof}
The proof of this proposition is along the lines of that of \cite[ Chapter XXI \textsection 2 Theorem 2]{DL6}, with replacement of their flow map for the constant velocity case with our more general one.
\end{proof}
In view of \cref{prop:BoundaryWithoutTau}, we observe that $D(A)\subset \widetilde W^p(\bbD_b)$, and \cref{eq:int_parts_p} holds for elements in $D(A)$.

Next, we treat the case of homogeneous boundary conditions, and, subsequently, we discuss nonhomogeneous boundary conditions by suitable liftings. The general case is then treated by superposition. 

\begin{lemma}
\label{lemma:SolutionBoundedHom}
Suppose that the data for problem \cref{prob:LiouvilleBoundNonHom1}--\cref{prob:LiouvilleBoundNonHom3} satisfies
\begin{align*}
    g=0,\qquad q\in L^p((0,T),L^p(\bbD_b)),\qquad u_0\in L^p(\bbD_b),\qquad p\in [1,\infty).
\end{align*}
Then problem \cref{prob:LiouvilleBoundNonHom1}--\cref{prob:LiouvilleBoundNonHom3} has a unique mild solution $u\in C^0([0,T],L^p(\bbD_b))$ given by
\begin{align}\label{eq:var_of_constants2}
    u(x,k,t) = (G(t)u_0)(x,k) + \int_0^t (G(t-s)q)(x,k,s)\ddd{s},\qquad 0\leq t\leq T,
\end{align}
with $G$ defined in \cref{eq:semigroupbounded}.
Furthermore, if the data additionally satisfies
\begin{align*}
    q\in C^1([0,T], L^p(\bbD_b)),\qquad u_0\in D(A).
\end{align*}
Then $u$ is a unique classical solution, i.e., $u\in C^1((0,T),L^p(\bbD_b))\cap C^0([0,T],D(A))$.
\end{lemma}
\begin{proof}
Using the infinitesimal generator $A$ defined in \cref{prop:semigroupbounded} we can rewrite \cref{prob:LiouvilleBoundNonHom1}--\cref{prob:LiouvilleBoundNonHom3} with $g=0$ as abstract Cauchy problem
\begin{align*}
    \frac{\partial u}{\partial t} &= Au+q,\qquad 0<t\leq T,\\
    u(0) &= u_0.
\end{align*}
The assertion then follows from standard results in semigroup theory, see, e.g., \cite[pp. 106-107]{Pazy1983}.
\end{proof}
We next discuss the case $q=0$, $u_0=0$.
\begin{lemma}\label{lem:existence_nonhom_bc}
Let $g\in C^1([0,T],\bbT_-^p)$ for $1\leq p<\infty$, and suppose $g(x_0,k_0,0)=0$ for a.e. $(x_0,k_0)\in\Gamma_-$.
Then the function 
\begin{align}\label{eq:tmp1}
    \widetilde\ME g(x,k,t) \colonequals g(\phi_{-\tau_-(x,k)}(x,k),t-\tau_-(x,k))\chi_{\bbU}(x,k)Y(t-\tau_-(x,k))
\end{align}
is in $C^0([0,T],W^p(\bbD_b))\cap C^1([0,T],L^p(\bbD_b))$ and
is a solution to \cref{prob:LiouvilleBoundNonHom1}--\cref{prob:LiouvilleBoundNonHom3} with $q=0$ and $u_0=0$.
\end{lemma}
\begin{proof}
We start by observing that $\widetilde\ME g$ can be written as follows
\begin{align*}
    \widetilde\ME g(\phi_s(x_0,k_0),t+s) = g(x_0,k_0,t)
\end{align*}
for $0\leq s\leq \min(t,\tau_+(x_0,k_0))$ and $(x_0,k_0)\in\Gamma_-$, and $\widetilde{\ME g}=0$ for points not reached by the curves $(\phi_s(x_0,k_0),t+s)$. Hence, $\widetilde \ME g$ is constant along the curves $s\mapsto (\phi_s(x_0,k_0),t+s)$, and, therefore
\begin{align*}
    \frac{\partial \widetilde\ME g}{\partial t}+[H,\widetilde\ME g]=0 \quad\text{on } \bbD_b\times (0,T).
\end{align*}
Moreover, $\widetilde\ME g=g$ on $\Gamma_-\times (0,T)$ and $\widetilde\ME g=0$ on $\bbD_b\times \{0\}$, and, therefore, $\widetilde\ME g$ satisfies \cref{prob:LiouvilleBoundNonHom1}--\cref{prob:LiouvilleBoundNonHom3} with $u_0=0$ and $q=0$.

It remains to show the claimed regularity properties of $\widetilde\ME g$.
Recalling that $\tau_t(x,k)=\min(\tau_+(x,k),t)$ and using \cref{eq:BoundaryTransform} and \cref{eq:tmp1}, a change of variables $\tilde s=s/\tau_t$ yields that
\begin{align*}
    \|\widetilde\ME g(t)\|_{L^p(\bbD_b)}^p&=\int_{\Gamma_-} \int_0^{\tau_+(x,k)} |g(x,k,t-s)|^p Y(t-s) \ddd{s} |\nabla_k H\cdot \vec{n}|\ddd{\Gamma}
    \\&
    =\int_{\Gamma_-} \int_0^{\min(t,\tau_+(x,k))} |g(x,k,s)|^p \ddd{s} |\nabla_k H\cdot \vec{n}|\ddd{\Gamma}\\
     &=\int_{\Gamma_-} \int_0^{1} |g(x,k,\tilde s)|^p \ddd{\tilde s} \tau_t(x,k)|\nabla_k H\cdot \vec{n}|\ddd{\Gamma}\\
     &=\int_0^1 \int_{\Gamma_-} |g(x,k,\tilde s)|^p  \tau_t(x,k)|\nabla_k H\cdot \vec{n}|\ddd{\Gamma}\ddd{\tilde s}\\
     &\leq \min(1,T)\sup_{0\leq t\leq T} \|g(t)\|_{\bbT_-^p}^p,
\end{align*}
where we used \cref{rem:traces} in the last step.
Since $\tau_t\leq \tau_T$, the previous identities also show that $t\mapsto \|\widetilde\ME g(t)\|_{L^p(\bbD_b)}$ is continuous for all $0\leq t\leq T$.
Let $\phi\in C^\infty_c(0,T)$. Using \cref{eq:tmp1} and $g(x_0,k_0,0)=0$ for all $(x_0,k_0)\in\Gamma_-$, we obtain that for a.e. $(x,k)\in\bbD_b$
\begin{align*}
    \int_0^T \widetilde \ME g(x,k,t)\frac{\partial\phi}{\partial t}(t)\ddd{t} &=- \int_{\tau_-(x,k)}^T   \left(\widetilde\ME\frac{\partial g}{\partial t}\right)(x,k,t) \phi(t)\ddd{t}\\
    &=- \int_{0}^T   \left(\widetilde\ME\frac{\partial g}{\partial t}\right)(x,k,t) \phi(t)\ddd{t}.
\end{align*}
Since $\frac{\partial g}{\partial t}\in C^0([0,T],\bbT_-^p)$ by assumption, we thus obtain that $t\mapsto \widetilde\ME g(x,k,t)$ is weakly differentiable and, using similar arguments as above, $\widetilde\ME \frac{\partial g}{\partial t} \in C^0([0,T],L^p(\bbD_b)$.
The assertion then follows from $[H,\widetilde\ME g]= -\frac{\partial \widetilde\ME g}{\partial t}$.
\end{proof}
Combining the results of \cref{lemma:SolutionBoundedHom} and \cref{lem:existence_nonhom_bc}, we obtain the following well-posedness result for  \cref{prob:LiouvilleBoundNonHom1}--\cref{prob:LiouvilleBoundNonHom3}.
\begin{theorem}
\label{thrm: SolutionNonHomogeneous}
Let $1\leq p <\infty$, and assume that
\begin{align*}
    q\in C^1([0,T], L^p(\bbD_b)),\quad u_0\in D(A),\quad
    g\in C^1([0,T],\bbT_-^p),\quad g(0) = 0.
\end{align*}
Then the function $u\in C^1((0,T),L^p(\bbD_b))\cap C^0([0,T], W^p(\bbD_b))$ given by
\begin{align}
    \label{eq:SolutionRepNonHom}
     u(x,k,t) = (G(t)u_0)(x,k)&+\int_0^t(G(t-s)q)(x,k,s)\ddd{s} 
      \\
      &+g(\phi_{-\tau_-(x,k)}(x,k),t-\tau_-(x,k))\chi_{\bbU}(x,k)Y(t-\tau_-(x,k)) \nonumber
\end{align}
is the unique solution to \cref{prob:LiouvilleBoundNonHom1}--\cref{prob:LiouvilleBoundNonHom3}.
\end{theorem}
\begin{proof}
Uniqueness of a solution follows from \cref{lemma:SolutionBoundedHom}  with $g=0$, $q=0$ and $u_0=0$.
To obtain existence of a solution, we use superposition. Let $w\in C^1((0,T),L^p(\bbD_b))\cap C^0([0,T],W^p(D(A))$ denote the solution constructed in \cref{lemma:SolutionBoundedHom}.
Then $u=w+\widetilde\ME g$, with $\widetilde\ME g$ defined in \cref{eq:tmp1}, has the claimed properties.
\end{proof}
\begin{remark}
 We note that \cref{eq:SolutionRepNonHom} is well-defined even for less regular data; for instance, $q\in L^p((0,T),L^p(\bbD_b))$, $u_0\in L^p(\bbD_b)$ and $g\in C^0([0,T],\bbT_-^p)$, gives $u\in C^0([0,T],L^p(\bbD_b))$.
 Therefore, \cref{eq:SolutionRepNonHom} can be used to define mild solutions for the inhomogeneous problem \cref{prob:LiouvilleBoundNonHom1}--\cref{prob:LiouvilleBoundNonHom3}. In fact, \cref{eq:SolutionRepNonHom} can also be derived from \cref{eq:var_of_constants2} and the usual argument of homogenization of the boundary conditions using the lifting \cref{eq:lifting} constructed in the proof of the trace lemma, see \cite{Schwenninger:2020} for a proper introduction of such concepts ensuring that the formula \cref{eq:SolutionRepNonHom} is independent of the employed lifting even for less regular data.
\end{remark}
\section{Well-posedness of the full RTE on bounded domains}\label{sec:existence_full_rte_bounded}
We now consider the full matrix radiative transfer equation in \cref{eq:rte_W} on $\bbD_b$.  Our approach follows that of \Cref{section: CoupleScatter}.
In the case of a bounded domain $\bbD_b$ the results on the coupling and scattering operator obtained in \Cref{section: CoupleScatter} remain true. In particular \cref{lemma:coupling} and \cref{proposition:Scattering} are still valid on $L^p(\bbD_b,\HM)$. Thus, the operator $\Sigma-S-N$ is bounded on $L^p(\bbD_b,\HM)$. It remains to define a suitable generator.

Recall that $[H,\MV]$ denotes the matrix-valued function obtained from applying the Poisson bracket to each entry of $\MV$.
Similar to \Cref{section: CoupleScatter}, we define the space 
\begin{align*}
W^p(\bbD_b,\HM) \colonequals \{ \MV \in L^p(\bbD_b,\HM) \colon [H,\MV] \in L^p(\bbD_b,\HM)\}.
\end{align*}
The infinitesimal generator $A$ defined in \cref{prop:semigroupbounded} for the scalar case can then be naturally extended to a generator $\MA$ for the matrix case as in \Cref{section: CoupleScatter}, with domain
\begin{align*}
    D(\MA) \colonequals \left\{\MV\in W^p(\bbD_b,\HM): \MV|_{\Gamma_-}=0\right\}.
\end{align*}
We denote, in slight abuse of notation, also the trace spaces for functions in $W^p(\bbD_b,\HM)$ with $ \bbT^p_\pm$ and $\widetilde \bbT^p_\pm$.
With this setting, well-posedness for the full matrix equation is established next.
\begin{theorem}
\label{thrm:WellPosednesBounded}
Let $1\leq p<\infty$ and suppose that
\begin{align*}
   \MF\in C^1((0,T),L^p(\bbD_b,\HM)),&\quad \MW_0\in D(\MA),\\
    {\MG}\in C^1((0,T), \bbT_-^p),&\quad \MG(0) = 0.
\end{align*}
Then the problem
\begin{alignat}{6}
    \frac{\partial \MW}{\partial t} + [H,\MW]-N(\MW) + \Sigma \MW&= S(\MW)+\MF &\quad& \text{on }\bbD_b\times(0,T),\label{eq:MatrixProblemNonHomogeneous1}\\
    \MW &= \MG & &\text{on }\Gamma_-\times(0,T), \label{eq:MatrixProblemNonHomogeneous2}\\
    \MW(0)&= \MW_0 &&\text{on }\bbD_b \label{eq:MatrixProblemNonHomogeneous3}
\end{alignat}
has a unique classical solution $\MW\in C^0([0,T],L^p(\bbD_b,\HM))\cap C^1((0,T),W^p(\bbD_b,\HM))$.
\end{theorem}
\begin{proof}
By extending \cref{lem:existence_nonhom_bc} to the matrix case, we can construct a matrix-valued lifting  $\MV\in C^0([0,T],W^p(\bbD_b,\HM))\cap C^1([0,T],L^p(\bbD_b,\HM))$ such that
\begin{alignat*}{6}
        \frac{\partial \MV}{\partial t}+ [H,\MV(t)]&= 0 &\quad& \text{on }\bbD_b \times (0,T),\\
        \MV &= \MG &&\text{on }\Gamma_-\times (0,T),\\
        \MV(0) &= 0 &&\text{on }\bbD_b.
\end{alignat*}
Next, consider the following problem with homogeneous boundary conditions,
\begin{alignat*}{4}
    \frac{\partial \MU}{\partial t} +[H,\MU]-N(\MU) + \Sigma \MU&= S(\MU)+\widetilde{\MF} &\quad& \text{on }\bbD_b\times(0,T),\\
    \MU &= 0 &&\text{on }\Gamma_-\times(0,T),\\
     \MU(0) &= \MW_0 &&\text{on }\bbD_b,
\end{alignat*}
with $\widetilde{\MF} = \MF+\Sigma\MV+N(\MV)-S(\MV) \in C^1((0,T),L^p(\bbD_b,\HM))$.
%
Since the statements of \cref{lemma:coupling} and \cref{proposition: ScatterBound} remain true for the bounded domain $\bbD_b$, the operator $\Sigma-S+N$ is a bounded perturbation of $\MA$ on $L^p(\bbD_b,\HM)$. Thus, there exists a unique classical solution $\MU\in C^0([0,T],D(\MA))\cap C^1((0,T),L^p(\bbD_b,\HM))$ to this problem \cite[p. 106]{Pazy1983}. Consequently, $\MW=\MU+\MV$ is a classical solution to \cref{eq:MatrixProblemNonHomogeneous1}--\cref{eq:MatrixProblemNonHomogeneous3} as claimed.
Uniqueness of the solution can be shown in the same manner as in \cref{thrm: SolutionNonHomogeneous}.
\end{proof}
\section{Properties of the coherence matrix}
\label{section: Additional}
With the well-posedness of the radiative transfer equation established, we here derive some additional properties of the solution. In the first subsection we derive an explicit representation of the time evolution of the solution norm and in the second subsection we show positivity of the solution.

\subsection{Conservation and dissipation of the norm of the solution}
\label{subsec72}
We start by considering the radiative transfer equation without interaction with the background medium ($S=\Sigma =\MF = 0)$. In this setting it can be shown that the solution norm is conserved up to in-, and outflow over the boundary of the domain.
\begin{theorem}
\label{thrm: NormNoScatter}
In addition to the assumptions of \cref{thrm:WellPosednesBounded} suppose that $\MG\in C^0([0,T],\widetilde\bbT_-^p)$, $\MF=0$, and $\Sigma=0$. The solution $\MW$ of \cref{eq:MatrixProblemNonHomogeneous1}--\cref{eq:MatrixProblemNonHomogeneous3} then satisfies
\begin{align*}
     \|\MW(t)\|^p_{L^p(\bbD_b,\HM)}+\int_0^t \|\MW(s)\|_{\widetilde\bbT_+^p}^{p} \ddd{s} 
    = \|\MW(0)\|^p_{L^p(\bbD_b,\HM)}+\int_0^t\|\widetilde\MW(s)\|_{\widetilde\bbT_-^p}^{p} \ddd{s}.
\end{align*}
\end{theorem}
\begin{proof}
Note that $\Sigma=0$ implies $S=0$ by \cref{eq:normalization}.
To compute the change of $\|\MW(t)\|_{S^p}^p$ over time, 
we express $\MW$ in terms of the Stokes parameters $(I,Q,U,V)$,
and recall that, cf. \cref{eq:sing_U},
\begin{align*}
    \|\MW\|^p_{L^p(\bbD_b,\HM)} 
    &= \int_{\bbD_b}\Bigg[\Big(\frac{1}{4}\big[I^2+Q^2+U^2+V^2]+\frac{1}{2}I\big[Q^2+U^2+V^2]^{1/2}\Big)^{p/2}\\
    &\quad+\Big(\frac{1}{4}\big[I^2+Q^2+U^2+V^2]-\frac{1}{2}I\big[Q^2+U^2+V^2]^{1/2}\Big)^{p/2}\Bigg]\dx\dk.
\end{align*}
The integrand is a function of the form $f(I,Q^2+U^2+V^2)$. Since, without the scattering and source terms, the components $I,Q,U$ and $V$ of the matrix $\MW$ are solutions to \cref{eq:evo_I}-\cref{eq:evo_V}, 
\cref{prop:CoupledFunction} implies
\begin{align*}
   \frac{d}{dt}\|\MW\|^p_{L^p(\bbD_b,\HM)} =\int_{\bbD_b}[\|\MW\|_{S^p}^p,H]\dx\dk.
\end{align*}
Since $\MW(t)_{\mid\Gamma_-}=\MG(t)_{\mid\Gamma_-}\in \widetilde\bbT_-^p$, 
\cref{prop:BoundaryWithoutTau} allows to use
formula \cref{eq:divergence}, which yields
\begin{align*}
    \frac{d}{dt}\|\MW\|^p_{L^p(\bbD_b,\HM)} =-\int_{\Gamma_+}\|\MW\|_{S^p}^{p}|\nabla_kH\cdot \vec{n}|\ddd{\Gamma}+\int_{\Gamma_-}\|\MW\|_{S^p}^{p}|\nabla_kH\cdot \vec{n}|\ddd{\Gamma},
\end{align*}
from which we obtain the assertion upon integration over $t$.
\end{proof}
Using \cref{proposition: ScatterBound},
it can be shown that the presence of scattering leads to dissipation.
\begin{lemma}
\label{thrm: NormEstimation}
In addition to the assumptions of \cref{thrm:WellPosednesBounded} suppose that $\MG\in C^0([0,T],\widetilde\bbT_-^p)$. The solution $\MW$ of \cref{eq:MatrixProblemNonHomogeneous1}--\cref{eq:MatrixProblemNonHomogeneous3} then satisfies
\begin{align*}
    \|\MW(t)\|^p_{L^p(\bbD_b,\HM)} +\int_0^t\|\MW(s)\|_{\widetilde\bbT_+^p}^{p}\ddd{s}
    &\leq \|\MW(0)\|^p_{L^p(\bbD_b,\HM)} +\int_0^t\|\MG(s)\|_{\widetilde\bbT_-^p}^{p}\ddd{s}\\
    &\quad+\int_{0}^t\|\MF(s)\|_{L^p(\bbD,\HM)} \|\MW(s)\|_{L^p(\bbD_b,\HM)}^{p-1}\ddd{s}.
\end{align*}
\end{lemma}
\begin{proof}
Observe that $\|\MW\|_{S^p}^p=\Tr(\MW\MW |\MW|^{p-2})$. 
We multiply \cref{eq:rte_W} by $\MW|\MW|^{p-2}$ and take traces. The terms involving the Poisson bracket $[H,\MW]$ and $N(\MW)$ can be treated as in the proof of \cref{thrm: NormNoScatter}.
Using \cref{proposition: ScatterBound} with $\MV=\MW|\MW|^{p-2}$, we further obtain that
\begin{align}\label{eq:dissipativity}
    \int_{\bbD_b}\Tr\big[\big(S(\MW)-\Sigma\MW\big)\MW|\MW|^{p-2}\big]\dx\dk \leq 0.
\end{align}
By applying Hölder's inequality, the integral over the source term can be written as
\begin{align*}
    \int_{\bbD_b}\Tr\big[\MF\MW|\MW|^{p-2}\big]\dx\dk \leq \int_{\bbD_b}\|\MF\|_{S^p}\|\MW\|_{S^p}^{p-1}\dx\dk\leq \|\MF\|_{L^p(\bbD_b,\HM)}\, \|\MW\|_{L^p(\bbD_b,\HM)}^{p-1}.
\end{align*}
Combining these inequalities and integrating over time then yields the assertion.
\end{proof}
\begin{remark}
 Inequality \cref{eq:dissipativity} shows that $S-\Sigma$ defines a dissipative operator on $L^p(\bbD_b,\HM)$. In view of \cref{proposition: ScatterBound}, a corresponding dissipativity result for $S-\Sigma$ holds true on $L^p(\bbD,\HM)$, i.e., for unbounded spatial domains. Inspecting the proof of \cref{thrm: NormNoScatter}, we also have shown that $S-\Sigma+N$ is dissipative.
\end{remark}


\subsection{Positivity of the coherence matrix}
To show positivity of the solution to problem \cref{eq:MatrixProblemNonHomogeneous1}--\cref{eq:MatrixProblemNonHomogeneous3} we first prove the following auxiliary result needed to deal with nonhomogeneous boundary conditions.
\begin{lemma}
\label{lemma: VNonHom}
Let $\MG$ satisfy the assumptions in \cref{thrm:WellPosednesBounded}, and additionally suppose that $\MG\succeq 0$. Then the problem
\begin{alignat}{6}
    \frac{\partial \MV}{\partial t} +[H,\MV]-N(\MV) + \Sigma \MV&= 0 &\quad& \text{on }\bbD_b\times(0,T), \label{eq:VNonHom1}\\
    \MV &=  \MG& &\text{on }\Gamma_-\times(0,T),\label{eq:VNonHom2}\\
    \MV(0) &= 0 &&\text{on }\bbD_b,\label{eq:VNonHom3}
\end{alignat}
has a unique positive solution $\MV\in C^0([0,T],L^p(\bbD_b,\HM))\cap C^1((0,T),W^p(\bbD_b,\HM))$.
\end{lemma}
\begin{proof}
The existence of a unique solution to problem \cref{eq:VNonHom1}--\cref{eq:VNonHom3} is guaranteed by \cref{thrm:WellPosednesBounded}. To show positivity of $\MV$, it suffices to show positivity of its trace and determinant.
The time evolution of the trace $t_\MV$ of $\MV$ is obtained by taking the trace of \cref{eq:VNonHom1}--\cref{eq:VNonHom3} and using that $\Tr\big[N(\MV)\big]=0$,
\begin{alignat*}{6}
   \frac{\partial}{\partial t}t_\MV +[H,t_\MV] &=  - \Sigma t_\MV &\quad& \text{on }\bbD_b\times(0,T), \\
     t_\MV &=  t_{\MG}& &\text{on }\Gamma_-\times(0,T),\\
    t_\MV(0) &= 0 &&\text{on }\bbD_b,
\end{alignat*}
where $t_{\MG}\geq 0$ denotes the trace of ${\MG}$.
Reasoning as in \cref{lem:existence_nonhom_bc}, we hence obtain
\begin{align*}
    t_\MV(x,k,t) =  E(x,k) \widetilde\ME t_{\MG}(x,k,t),
\end{align*}
with $E(x,k)=\exp\left(-\int_0^{\tau_-(x,k)}\Sigma(\phi_{-s}(x,k))\ddd{s}\right)$,
which shows $t_\MV(x,k,t)\geq 0$.

Positivity of the determinant of $\MV$ can be shown in a similar way, as explained next. Recall Jacobi's formula,
\begin{align*}
    \frac{\partial}{\partial t}\det(\MV) = \Tr\Big[\text{adj}(\MV)\frac{
    \partial \MV}{\partial t}\Big],
\end{align*}
with adjunct matrix $\text{adj}(\MV)=J^*\MV J$.
By multiplying \cref{eq:VNonHom1} with ${\rm adj}(\MV)$, taking traces, and using Jacobi's formula, the problem \cref{eq:VNonHom1}--\cref{eq:VNonHom3} is transformed
to the following boundary value problem for the determinant $d_\MV$ of $\MV$,
\begin{alignat*}{6}
    \frac{\partial}{\partial t}d_\MV + [H,d_\MV]&=-2\Sigma d_\MV &\quad&\text{on }\bbD_b\times(0,T),\\
    d_\MV &= d_{\MG}& &\text{on }\Gamma_-\times(0,T),\\
    d_\MV(0) &= 0 & &\text{on }\bbD_b,
\end{alignat*}
where $d_{\MG}\geq 0$ denotes the determinant of $\MG$.
With similar arguments used to discuss $t_\MV$, we conclude that $d_\MV(x,k,t)\geq 0$.
\end{proof}
We are now in the position to show positivity of the coherence matrix $\MW$, which involves also the discussion of the scattering term $S(\MW)$. 
\begin{theorem}
Let $\MF,\MW_0$ and $\MG$ satisfy the assumption made in \cref{thrm:WellPosednesBounded} and suppose that $\MF, \MW_0$ and $\MG$ are positive. Furthermore, denote by $\MW$ the solution to problem \cref{eq:MatrixProblemNonHomogeneous1}--\cref{eq:MatrixProblemNonHomogeneous3}. Then $\MW(t)$ is positive for all $t\geq 0$.
\end{theorem}
\begin{proof}
Denote $\MV$ the positive solution to \cref{eq:VNonHom1}--\cref{eq:VNonHom3} constructed in \cref{lemma: VNonHom}.
The function $\MU=\MW-\MV$ then satisfies
\begin{alignat*}{6}
    \frac{\partial \MU}{\partial t} +[H,\MU]-N(\MU) + \Sigma \MU&= S(\MU)+\widetilde{\MF} &\quad& \text{on }\bbD_b\times(0,T),\\
    \MU &= 0 &&\text{on }\Gamma_-\times(0,T),\\
    \MU(0) &= \MW_0 &&\text{on }\bbD_b,
\end{alignat*}
with $\widetilde{\MF} = \MF + S(\MV)$. Since $\MF$ is positive and $S(\MV)$ is positive by \cref{lemma:ScatterPositivity}, we have that $\widetilde{\MF}$ is positive. By using Trotter's formula as in the proof of \cref{thrm: ExistenceFull}, it follows that $\MU$ is positive for $\MW_0\succeq 0$. Thus, $\MW=\MU+\MV$ is also positive.
\end{proof}
\section{Conclusions}
\label{section: conclusion}
In this paper we have studied the properties of the general radiative transfer equation with polarization and varying refractive index. We established well-posedness under mild assumptions on the parameters via semigroup theory. In order to be able to handle bounded domains, which is the typical case considered in applications, we identified suitable trace spaces, and proved well-posedness of the radiative transfer equation on bounded spatial domains. Moreover, we have shown that the solution satisfies certain physical properties. For instance, the coherence matrix is Hermitian and positive pointwise in phase-space and for all times if the initial and boundary conditions are.
Moreover, for unbounded domains we showed that energy is dissipated only through interaction with the background medium, i.e., scattering; while for bounded domains, energy is additionally dissipated over the outflow boundary.


\section*{Acknowledgments}
VB and MS acknowledge support by the Dutch Research council (NWO) via the Mathematics Clusters grant no. 613.009.133.
\bibliographystyle{siamplain}
\bibliography{bib}

\begin{thebibliography}{10}

\bibitem{Arnush_1972}
{\sc D.~Arnush}, {\em Underwater light-beam propagation in the
  small-angle-scattering approximation}, Journal of the Optical Society of
  America, 62 (1972), p.~1109, \url{https://doi.org/10.1364/josa.62.001109}.

\bibitem{ArridgeSchotland:2009}
{\sc S.~R. Arridge and J.~C. Schotland}, {\em Optical tomography: forward and
  inverse problems}, Inverse Problems, 25 (2009), pp.~123010, 59,
  \url{https://doi.org/10.1088/0266-5611/25/12/123010}.

\bibitem{Bal:2009}
{\sc G.~Bal}, {\em Inverse transport theory and applications}, Inverse
  Problems, 25 (2009), pp.~053001, 48,
  \url{https://doi.org/10.1088/0266-5611/25/5/053001}.

\bibitem{Bhatia1997}
{\sc R.~Bhatia}, {\em Symmetric norms}, in Matrix Analysis, Springer New York,
  1997, pp.~84--111, \url{https://doi.org/10.1007/978-1-4612-0653-8_4}.

\bibitem{Born:1999}
{\sc M.~Born, E.~Wolf, A.~B. Bhatia, P.~C. Clemmow, D.~Gabor, A.~R. Stokes,
  A.~M. Taylor, P.~A. Wayman, and W.~L. Wilcock}, {\em Principles of Optics},
  Cambridge University Press, oct 1999.

\bibitem{Burger:97}
{\sc T.~Burger, J.~Kuhn, R.~Caps, and J.~Fricke}, {\em Quantitative
  determination of the scattering and absorption coefficients from diffuse
  reflectance and transmittance measurements: Application to pharmaceutical
  powders}, Appl. Spectrosc., 51 (1997), pp.~309--317.

\bibitem{Carminati_2021}
{\sc R.~Carminati and J.~C. Schotland}, {\em Principles of Scattering and
  Transport of Light}, Cambridge University Press, jun 2021,
  \url{https://doi.org/10.1017/9781316544693}.

\bibitem{Case1967}
{\sc K.~Case and P.~Zweifel}, {\em Linear Transport Theory}, Addison-Wesley
  Publishing Co., 1967.

\bibitem{Chandrasekhar1960}
{\sc S.~Chandrasekhar}, {\em {Radiative Transfer}}, Radiative Transfer, Dover
  Publications,  (1960).

\bibitem{Stefanov:1999}
{\sc M.~Choulli and P.~Stefanov}, {\em An inverse boundary value problem for
  the stationary transport equation}, Osaka J. Math., 36 (1999), pp.~87--104.

\bibitem{DL5}
{\sc R.~Dautray and J.-L. Lions}, {\em Mathematical analysis and numerical
  methods for science and technology. {V}ol. 5}, Springer-Verlag, Berlin, 1992.

\bibitem{DL6}
{\sc R.~Dautray and J.-L. Lions}, {\em Mathematical analysis and numerical
  methods for science and technology. {V}ol. 6}, Springer-Verlag, Berlin, 1993.
\newblock Evolution problems. II, With the collaboration of Claude Bardos,
  Michel Cessenat, Alain Kavenoky, Patrick Lascaux, Bertrand Mercier, Olivier
  Pironneau, Bruno Scheurer and R\'{e}mi Sentis, Translated from the French by
  Alan Craig.

\bibitem{Fasano2006}
{\sc A.~Fasano and S.~Marmi}, {\em Analytical Mechanics : an Introduction},
  Oxford University Press, 2006.

\bibitem{Gerard:1997}
{\sc P.~G{\'e}rard, P.~A. Markowich, N.~J. Mauser, and F.~Poupaud}, {\em
  Homogenization limits and wigner transforms}, Communications on Pure and
  Applied Mathematics, 50 (1997), pp.~323--379,
  \url{https://doi.org/10.1002/(sici)1097-0312(199704)50:4<323::aid-cpa4>3.0.co;2-c}.

\bibitem{Hale1980}
{\sc J.~K. Hale}, {\em Ordinary differential equations}, Robert E. Krieger
  Publishing Co., Inc., Huntington, N.Y., second~ed., 1980.

\bibitem{Hansen_1974}
{\sc J.~E. Hansen and L.~D. Travis}, {\em Light scattering in planetary
  atmospheres}, Space Science Reviews, 16 (1974), pp.~527--610,
  \url{https://doi.org/10.1007/bf00168069}.

\bibitem{Jensen2004}
{\sc M.~Jensen}, {\em Discontinuous {G}alerkin methods for {F}riedrichs systems
  with irregular solutions}, PhD thesis, Corpus Christi College, University of
  Oxford, 2004,
  \url{http://sro.sussex.ac.uk/id/eprint/45497/1/thesisjensen.pdf}.

\bibitem{Jiang:1997}
{\sc Z.~Jiang}, {\em On the {L}iouville equation}, Wuhan Univ. J. Nat. Sci., 2
  (1997), pp.~273--275, \url{https://doi.org/10.1007/BF02829903}.

\bibitem{Lieb1976}
{\sc E.~Lieb and W.~Thirring}, {\em Inequalities for the Moments of the
  Eigenvalues of the Schr{\"o}dinger Hamiltonian and Their Relation to Sobolev
  Inequalities}, Princeton Press, 1976, pp.~301--302.

\bibitem{Manteuffel1999}
{\sc T.~A. Manteuffel, K.~J. Ressel, and G.~Starke}, {\em A boundary functional
  for the least-squares finite- element solution of neutron transport
  problems}, SIAM Journal on Numerical Analysis, 37 (1999), pp.~556--586,
  \url{https://doi.org/10.1137/s0036142998344706}.

\bibitem{Melikov_2018}
{\sc R.~Melikov, D.~A. Press, B.~G. Kumar, S.~Sadeghi, and S.~Nizamoglu}, {\em
  Unravelling radiative energy transfer in solid-state lighting}, Journal of
  Applied Physics, 123 (2018), p.~023103,
  \url{https://doi.org/10.1063/1.5008922}.

\bibitem{Papanicolaou1975}
{\sc G.~C. Papanicolaou and R.~Burridge}, {\em {Transport equations for the
  Stokes parameters from Maxwell's equations in a random medium}}, Journal of
  Mathematical Physics, 16 (1974), pp.~2074--2085,
  \url{https://doi.org/10.1063/1.522422}.

\bibitem{Pazy1983}
{\sc A.~Pazy}, {\em {Semigroups of Linear Operators and Applications to Partial
  Differential Equations}}, Springer New York, 1983,
  \url{https://doi.org/10.1007/978-1-4612-5561-1}.

\bibitem{Petrina_1990}
{\sc D.~Y. Petrina and V.~I. Gerasimenko}, {\em Mathematical problems of
  statistical mechanics of a system of elastic balls}, Russian Mathematical
  Surveys, 45 (1990), pp.~153--211,
  \url{https://doi.org/10.1070/rm1990v045n03abeh002360}.

\bibitem{Ryzhik:1996}
{\sc L.~Ryzhik, G.~Papanicolaou, and J.~B. Keller}, {\em Transport equations
  for elastic and other waves in random media}, Wave Motion, 24 (1996),
  pp.~327--370, \url{https://doi.org/10.1016/S0165-2125(96)00021-2}.

\bibitem{Schwenninger:2020}
{\sc F.~L. Schwenninger}, {\em Input-to-state stability for parabolic boundary
  control:linear and semilinear systems}, in Control Theory of
  Infinite-Dimensional Systems, J.~Kerner, H.~Laasri, and D.~Mugnolo, eds.,
  Cham, 2020, Springer International Publishing, pp.~83--116.

\bibitem{Teschl2004}
{\sc G.~Teschl}, {\em {Ordinary differential equations and Dynamical Systems}},
  vol.~140, American Mathematical Society, Providence, Rhode Island, aug 2004,
  \url{https://doi.org/10.1090/gsm/140}.

\end{thebibliography}
\end{document}